\documentclass[a4paper]{amsart}
\author{Dilip Raghavan}
\thanks{First author was partially supported by the Singapore Ministry of Education's research grant number MOE2017-T2-2-125.}
\address{Department of Mathematics\\
National University of Singapore\\
Singapore 119076.}
\email{raghavan@math.nus.edu.sg}
\urladdr{http://www.math.nus.edu.sg/$\sim$raghavan}
\author{Stevo Todorcevic}
\thanks{Second author is partially supported by grants from NSERC (455916) and CNRS (IMJ-PRG UMR7586).}
\address{Department of Mathematics, University of Toronto, Toronto, Canada, M5S 2E4.}
\address{Institut de Math\'{e}matique de Jussieu, UMR 7586, Case 247, 4 place Jussieu, 75252 Paris Cedex, France.}
\email{stevo@math.toronto.edu, todorcevic@math.jussieu.fr}
\date{\today}
\title{Proof of a Conjecture of Galvin}
\usepackage{amssymb, amsmath, amsthm, mathrsfs, enumerate, enumitem, amsfonts, latexsym, bbm}
\def\polhk#1{\setbox0=\hbox{#1}{\ooalign{\hidewidth
    \lower1.5ex\hbox{`}\hidewidth\crcr\unhbox0}}}
\newtheorem{Theorem}{Theorem}

\newtheorem{Lemma}[Theorem]{Lemma}
\newtheorem{Cor}[Theorem]{Corollary}
\newtheorem*{gp}{General Problem}

\theoremstyle{definition}
\newtheorem{Def}[Theorem]{Definition}

\theoremstyle{remark}

\newcommand{\restrict}{\mathord{\upharpoonright}}

\renewcommand{\[}{\left[}
\renewcommand{\]}{\right]}

\newcommand{\Q}{\mathbb{Q}}
\newcommand{\R}{\mathbb{R}}
\newcommand{\lc}{\left|}
\newcommand{\rc}{\right|}
\newcommand{\wor}{{<}_{\mathtt{wo}}}

\DeclareMathOperator{\dom}{dom}

\DeclareMathOperator{\rk}{rank}

\newcommand{\Pset}{\mathcal{P}}

\newcommand{\nn}{\mathbb{N}}
\newcommand{\zz}{\mathbb{Z}}

\newcommand{\BB}{\mathcal{B}}
\newcommand{\CC}{\mathcal{C}}

\newcommand{\GG}{{\mathcal{G}}}

\newcommand{\I}{{\mathcal{I}}}

\newcommand{\F}{{\mathcal{F}}}

\newcommand{\TT}{{\mathcal{T}}}
\newcommand{\KK}{{\mathcal{K}}}

\newcommand{\pr}[2]{\langle #1, #2 \rangle}
\newcommand{\lex}[2]{#1 \: {<}_{\mathord{\mathrm{lex}}} \: #2}
\newcommand{\seq}[4]{\langle {#1}_{#2}: #2 #3 #4 \rangle}
\begin{document}
\begin{abstract}
We prove that if the set of unordered pairs of real numbers is colored by finitely many colors, there is a set of reals homeomorphic to the rationals whose pairs have at most two colors.
Our proof uses large cardinals and it verifies a conjecture of Galvin from the 1970s.
We extend this result to an essentially optimal class of topological spaces in place of the reals.
\end{abstract}
\maketitle
\section{Introduction} \label{sec:intro}
In this paper we present a result that sheds light on a general problem about the behavior of an arbitrary relational structure of the form $(\mathbb{R}, S_1,...,S_n)$ on `large' subsets of $\mathbb{R}.$
A general result of Ehrenfeucht--Mostowski~\cite{EM}, anticipated already in the seminal paper of Ramsey \cite{ramsey}, shows that such problems can be reduced to problems about finite colorings of the symmetric cubes ${[\mathbb{R}]}^{k}$ (the set of all $k$-element sets of real numbers), where the integer $k$ is closely related to the arity of the (finite list of) relations of the given structure on $\mathbb{R}.$
In other words, in our general problem we could restrict ourselves to relational structures of the form $(\mathbb{R}, E)$, where $E$ is a single equivalence relation with finitely many equivalence classes on an appropriate symmetric cube ${[\mathbb{R}]}^{k}.$
Answering a question of Knaster, in 1933, Sierpi{\' n}ski~\cite{Sierp} has shown that a well-ordering $\wor$ of $\mathbb{R}$ can be used in defining a particular equivalence relation ${E}_{k}^{S}$ on the finite symmetric cube ${[\mathbb{R}]}^{k}$ with $k!(k-1)!$ classes by comparing the behaviors of the well-ordering $\wor$ and the usual ordering on a given $k$-element set $s$ as well as recording the ordering of distances between consecutive elements of $s$ when enumerated increasingly according to the usual ordering of $\mathbb{R}.$
What Sierpi{\' n}ski's proof shows is that the number $k!(k-1)!$ of equivalence classes of ${E}_{k}^{S}$ cannot be reduced by restricting it to any uncountable, or more generally, nonempty and dense in itself subset of $\mathbb{R}.$
This feature of Sierpi{\' n}ski's proof was first put forward by Galvin in a letter to Laver (\cite{galvinletter}), and it was reiterated few years later when Baumgartner proved that in this problem $\mathbb{R}$ cannot be replaced by any countable topological space. Baumgartner~\cite{baumtop} explicitly states the $2$-dimensional version of Galvin's conjecture solved here, with an opinion that this is probably the most interesting open problem in this area.
More precisely, we show using large cardinals that if $X$ is an arbitrary uncountable set of reals and $E$ is an equivalence relation on ${[X]}^{2}$, then there is $Y \subseteq X$ homeomorphic to $\mathbb{Q}$ such that $E \restrict {[Y]}^{2}$ is coarser than ${E}_{2}^{S} \restrict {[Y]}^{2}.$
In fact we shall isolate what appears to be the optimal general topological condition on the space $X$ that guarantees this conclusion with ${E}_{2}^{S}$ replaced by an appropriate equivalence relation on ${[X]}^{2}$ that has exactly $2$ classes when restricted to any topological copy of $\mathbb{Q}$ inside $X.$

We finish this introduction with comments on the methods behind the proofs of these results.
Given a space $X$ satisfying certain conditions and a finite coloring  $c: {[X]}^{2} \rightarrow l$, we use large cardinals to construct a topological copy $Y \subseteq X$ of $\Q$ such that ${[Y]}^{2}$ uses no more than 2 colors.
In hindsight the conditions on $X$ are made in order to allow us a construction using large cardinals of another space $Z$ together with a continuous map $f: Z \rightarrow X$ such that $Z$ is a Baire space, that $f$ is not constant on any nonempty open subset of $X$, and that the induced coloring ${c}_{f}: {[Z]}^{2} \rightarrow l+1$  (given by ${c}_{f}(x,y)=c(f(x), f(y))$ if $f(x)\neq f(y)$ and ${c}_{f}(x,y)=l$ if $f(x)=f(y)$) is in some sense Baire measurable.
Thus the problem is transferred to $Z$ where it becomes possible to use Banach-Mazur games to construct a copy of $\Q$ which uses only two colors of ${c}_{f}$ and on which $f$ is one-to-one.
The conditions on $X$ which allow us (using large cardinals) such transfer to a Baire space $Z$ and a continuous nowhere constant map $f$ had been already used in the paper~\cite{genericcontinuity}, which in turn was motivated by a problem of Haydon~\cite{haydon90} from the theory of differentiability in the context of general Banach spaces.
It should also be noted that large cardinals are introduced into the construction of $Z$ and $f: Z \rightarrow X$ through the ideas behind the stationary tower forcing of Woodin \cite{woodin88}, which in turn was inspired by the groundbreaking work of Foreman, Magidor and Shelah~\cite{fms1}.
We believe that applying large cardinals to structural Ramsey theory is a new idea that will give us more results of this kind.
In fact, we are now investigating if this idea will also lead us to the proof of the higher-dimensional version of Galvin's Conjecture stating that for every integer $k\geq 2,$ an arbitrary coloring of 
${[\R]}^{k}$ can be reduced to the Sierpi{\' n}ski coloring on a topological copy of $\Q.$
Finally, we mention that the precise forms of our results are explained in Sections \ref{intro} and \ref{intro2} where we comment on their general interest and how they are related to other areas of mathematics.
\section{Ramsey degree calculus}\label{intro}
In this section we state the general form of our result for sets of reals, putting it into the context of other results in this area.
One of the goals of Ramsey theory is to find the smallest number of colors that must necessarily occur among the pairs in any sufficiently rich substructure of a more complicated structure whenever all of the pairs from the more complicated structure have been colored with finitely many colors.
More generally, suppose that $\CC$ is come class of structures and that $A$ is a structure that embeds into every member of $\CC$.
For each natural number $k \geq 1$, we would like to know the smallest number ${t}_{k}$ such that for every natural number $l \geq 1$, for every structure $B \in \CC$, and for every coloring that assigns one of $l$ colors to each $k$-element subset of $B$, there exists a substructure $X \subseteq B$ which is isomorphic to $A$ and has the property that at most ${t}_{k}$ colors occur among the $k$-element subsets of $X$.  
This natural number ${t}_{k}$, if it exists, is called the \emph{$k$-dimensional Ramsey degree of the structure $A$ within the class $\CC$}.
Determining this number produces a finite basis for the class of all colorings that assign one of finitely many colors to each $k$-element subset of some structure $B \in \CC$, in the sense that it shows that an arbitrary such coloring is equivalent to one of finitely many canonical colorings, when all colorings are restricted to a substructure of $B$ that is isomorphic to $A$.  
This finite list of canonical colorings can be determined once ${t}_{k}$ is computed.

The problem of computing the Ramsey degrees of $A$ in a class of structures $\CC$ can be formulated as an expansion problem.
Let us say that \emph{$R$ is a finitary relation on $A$} to mean that there is an integer $k \geq 1$ so that $R$ consists of sequences of length $k$ in $A$.
Solving the expansion problem for $A$ within $\CC$ requires identifying a list of finitely many finitary relations ${R}_{1}, \dotsc, {R}_{n}$ on $A$ that are atomic for the structures in $\CC$ in the following sense: for each structure $B \in \CC$ and an arbitrary finitary relation $S$ on $B$, there must exist a substructure $X \subseteq B$ and an isomorphism $\varphi: A \rightarrow X$ such that the restriction of $S$ to $X$ is definable without quantifiers from the images of ${R}_{1}, \dotsc, {R}_{n}$ under $\varphi$.
Determining the Ramsey degrees of $A$ within $\CC$ solves the expansion problem for $A$ within $\CC$.
Frequently, the atomic finitary relations that solve the expansion problem turn out to be purely order-theoretic in nature.    
An example of such a computation of canonical forms for arbitrary finitary relations on $\nn$ via Ramsey's original theorem can be found in Chapter 1 of \cite{introramsey} (Theorem 1.7 of \cite{introramsey}).
The binary relations $<, =$, and $>$ are the only atomic relations needed to define an arbitrary finitary relation on $\nn$ without the help of quantifiers, once everything has been restricted to a suitable isomorphic copy of $\nn$.
This computation of canonical forms for relations on $\nn$ was originally done by Ramsey in \cite{ramsey}; a closely related result was rediscovered by Erd{\H o}s and Rado~\cite{erdradeq}.
It should be clear that the richer the structure of $A$ is the more informative is a solution to the expansion problem for $A$ in $\CC$.

Expansion problems for various pairs $\pr{A}{\CC}$ occur frequently in topological dynamics in the form of questions about representations of the universal minimal flow of the automorphism group of an ultrahomogeneous structure.
See \cite{kpt} for further details on the connections between Ramsey theory and topological dynamics of automorphism groups where a precise correspondence is given between Ramsey degree calculus and representation theory for universal minimal flows of such groups.

In this paper, it will be proved, assuming large cardinals, that the $2$-dimensional Ramsey degree of the topological space $\Q$ of the rationals within the class of all regular, non-left-separated spaces with a point countable base is at most $2$.
Our result is provably optimal for metrizable spaces.
The following terminology will make certain results easier to state.
\begin{Def} \label{def:degree}
 Let $X$ be any set.
 For any cardinal number $\kappa$, ${\[X\]}^{\kappa}$ is the collection of subsets of $X$ of cardinality $\kappa$, and ${\[X\]}^{< \kappa}$ denotes the collection of subsets of $X$ that have cardinality less than $\kappa$.
 
 Let $X$ and $Y$ be topological spaces.
 For natural numbers $k, l, t \geq 1$, we will write
 \begin{align*}
 X \rightarrow {\left( Y \right)}^{k}_{l, t}
 \end{align*}
 to mean that for every set $L$ of cardinality $l$ and every coloring $c: {\[X\]}^{k} \rightarrow L$, there exist a subspace $Y' \subseteq X$  homeomorphic to $Y$ and a subset $T \subseteq L$  of cardinality $t$ such that $\left\{c(v): v \in {\[Y'\]}^{k}\right\} \subseteq T$.
 If $t=1,$ then it is not recorded in this notation, i.e.\@, we write $X \rightarrow {\left(  Y \right)}^{k}_{l}$ instead of $X \rightarrow {\left(  Y \right)}^{k}_{l, 1}.$
 
 For a natural number $k \geq 1$, the \emph{$k$-dimensional Ramsey degree of a space $Y$ inside the space $X$}, if it exists, is the least natural number $t \geq 1$ with the property that $X \rightarrow {\left( Y \right)}^{k}_{l, t}$  for all $l < \omega.$ 
\end{Def}
Using the terminology of Definition \ref{def:degree}, one of the important consequences of the main result of this paper may be stated as follows.
\begin{Theorem} \label{thm:metricspaces}
 Assume either that there is a proper class of Woodin cardinals or an uncountable strongly compact cardinal.
 Let $X$ be a non $\sigma$-discrete metric space.
 Then the $2$-dimensional Ramsey degree of $\Q$ in $X$ is at most 2.
\end{Theorem}
That non $\sigma$-discreteness is an optimal restriction in this theorem follows from the result below.
\begin{Theorem}[\cite{todorcevicweissnotes}] \label{thm:todorcevicweiss}
If $X$ is a $\sigma$-discrete metric space, then there is $c: {[X]}^{2} \rightarrow \omega$ such that $c''{[Y]}^{2} = \omega$ for all $Y \subseteq X$ homeomorphic to $\Q.$
\end{Theorem}
It follows that the Ramsey degree of $\Q$ does not exist (is infinite) in any $\sigma$-discrete metric space.
The equivalence stated in the following corollary encapsulates Theorem \ref{thm:metricspaces} and the fact that it is optimal for metrizable spaces.
\begin{Cor} \label{cor:general}
 Assume either that there is a proper class of Woodin cardinals or an uncountable strongly compact cardinal.
 Then the following are equivalent for every metrizable space $X:$
 \begin{enumerate}
  \item[(a)]
  $X$ is not $\sigma$-discrete;
  \item[(b)]
  $X \rightarrow {\left( \Q \right)}^{2}_{l, 2}$ for every natural number $l \geq 1$.
 \end{enumerate}
\end{Cor}
Note that condition (a) is equivalent to $X \rightarrow {\left( \omega+1 \right)}^{1}_{\omega}, $ so we have here an interesting analogy between Corollary \ref{cor:general} and a theorem of Todorcevic from \cite{posetpartition} stating that for any partial order $P$, $P \rightarrow {\left( \omega \right)}^{1}_{\omega}$ if and only if $P \rightarrow {\left( \alpha \right)}^{2}_{k}$ for all $\alpha < {\omega}_{1}$ and $k < \omega$.

The special case of Theorem \ref{thm:metricspaces} restricted to uncountable sets $X \subseteq \R$ is particularly interesting since in this case we have a coloring $s:{[X]}^{2} \rightarrow 2$ which witnesses $X \not\rightarrow {\left( \Q \right)}^{2}_{2},$ i.e.\@ that the Ramsey degree of $\Q$ in $X$ is at least 2, and therefore equal to 2.
Recall how Sierpi{\' n}ski's coloring is defined using a well-ordering of the reals $\wor$ and the usual ordering $<$. 
Define $s: {\[\R\]}^{2} \rightarrow \{0, 1\}$ by stipulating that $s(\{x, y\}) = 0$ if and only if $\wor$ and $<$ agree on $\{x, y\}$, for all pairs $\{x, y\} \in {\[\R\]}^{2}$.
To see that this coloring establishes $\R \not\rightarrow {\left( \Q \right)}^{2}_{2},$ note that any monochromatic subset of $\R$ must either be well-ordered or reverse well-ordered by $<$.
Hence no subset of $\R$ which contains a $\zz$-chain in the usual ordering can be monochromatic.
Let ${E}^{S}$ be the equivalence relation on ${[\R]}^{2}$ that has the two sets ${s}^{-1}(i) (i<2)$ as equivalence classes.
A single Woodin cardinal is sufficient to prove the restriction of Theorem~\ref{thm:metricspaces} to uncountable sets $X \subseteq \R$.
\begin{Cor}\label{cor:basis}
Assume either that there is a Woodin cardinal or an uncountable strongly compact cardinal.
Let $X$ be an uncountable set of reals.
Then for every equivalence relation $E$ on ${[X]}^{2}$ with finitely many equivalence classes, there is $Y \subseteq X$ homeomorphic to $\Q$ such that $E\restrict {\[Y\]}^{2}$ is coarser than
${E}^{S}\restrict {\[Y\]}^{2}$. 
\end{Cor}
\begin{proof}
To see this, let $l \geq 1$ be a natural number and let $c: {\[X\]}^{2} \rightarrow l$ be a coloring which is giving us an equivalence relation on ${[X]}^{2}$ with $l$ classes.
Define a new coloring $d: {\[X\]}^{2} \rightarrow l \times 2$ by setting $d(\{x, y\}) = \pr{c(\{x, y\})}{s(\{x, y\})}$, for every $\{x, y\} \in {\[X\]}^{2}$.
Here $s$ is Sierpi{\' n}ski's coloring defined above from an arbitrary well-ordering of $\R$.
Applying Theorem \ref{thm:metricspaces}, there must be a set $Y \subseteq X$ as well as colors $i, j < l$ such that $Y$ is homeomorphic to $\Q$ and $\left\{c(v): v \in {\[Y\]}^{2}\right\} \subseteq \{\pr{i}{0}, \pr{j}{1}\}$.
If $i = j$, then $c$ is constant on ${\[Y\]}^{2}$.
And if $i \neq j$, then $c$ is equivalent to $s$ on ${\[Y\]}^{2}$, with the color $i$ playing the role of the color $0$ of $s$ and $j$ playing the role of $1$.
\end{proof}
Theorem \ref{thm:metricspaces} also implies that any well-ordering $\wor$ solves the $2$-dimensional expansion problem for $\Q$ within the class of all uncountable sets of reals.
\begin{Cor} \label{cor:basisandexpansion}
 Assume either that there is a Woodin cardinal or an uncountable strongly compact cardinal.
 Let $\wor$ be any well-ordering of $\R$ and let $<$ be the usual ordering of $\R$.
 Then for every uncountable $X \subseteq \R$ and every binary relation $M \subseteq {X}^{2}$, there exists a set $Y \subseteq X$, which is homeomorphic to $\Q$, such that $M \cap {Y}^{2}$ is equal to one of the following relations restricted to $Y:$ $\top,$ $\bot,$ $=,$ $\neq,$ $<,$ $>,$ $\leq,$ $\geq,$ $\wor,$ ${>}_{\mathtt{wo}},$ ${\leq}_{\mathtt{wo}},$ ${\geq}_{\mathtt{wo}},$
 $< \cap \: \wor,$ $< \cap \: {>}_{\mathtt{wo}},$ $> \cap \: \wor,$ $> \cap \: {>}_{\mathtt{wo}},$ $\leq \cap \: {\leq}_{\mathtt{wo}},$ $\leq \cap \: {\geq}_{\mathtt{wo}},$ $\geq \cap \: {\leq}_{\mathtt{wo}},$ and $\geq \cap \: {\geq}_{\mathtt{wo}}.$
\end{Cor}
\begin{proof}
This is similar to the proof of Theorem 1.7 in \cite{introramsey} where in the crucial step we use Theorem \ref{thm:metricspaces} in place of Ramsey's theorem.
\end{proof}
A weak form of the conclusion of Theorem \ref{thm:metricspaces} when we restrict the class of all non $\sigma$-discrete metric spaces to the singleton $\{\R\}$ was first conjectured by Galvin in the 1970s (\cite{galvinletter}), and Galvin's conjecture, even in this weak form, remained unproved until our work.
In an earlier unpublished note, Galvin had proved that for every coloring of ${\[\Q\]}^{2}$ into finitely many colors, there exists a $Y \subseteq \Q$ which is order-isomorphic to $\Q$ such that at most $2$ colors occur in ${\[Y\]}^{2}$.
This was generalized by Laver, who showed that for each natural number $k \geq 1$, there exists a number ${t}_{k}$ with the property that for every coloring of ${\[\Q\]}^{k}$ into finitely many colors, there exists a $Y \subseteq \Q$ which is order-isomorphic to $\Q$ such that at most ${t}_{k}$ colors occur in ${\[Y\]}^{2}$.
The optimal value of ${t}_{k}$ was computed by Devlin~\cite{devlind}.
He showed that ${t}_{k} = {T}_{2k-1}$ is the minimal natural number that witnesses Laver's result, where $\tan(x) = {\sum}_{n = 0}^{\infty}{\frac{{T}_{n}}{n!}{x}^{n}}$.
Furthermore, Devlin's theorems produce for each natural number $k \geq 1$ a finite list of canonical colorings of ${\[\Q\]}^{k}$ into at most ${T}_{2k-1}$ colors such that an arbitrary coloring of ${\[\Q\]}^{k}$ into any finite number of colors is equivalent to one of the canonical colorings on an order isomorphic substructure of $\Q$. 
A recent exposition of Devlin's work can be found in Chapter 6 of \cite{introramsey}.

Baumgartner~\cite{baumtop} was the first to prove  that there is a significant difference when the topological structure of $\Q$ is considered instead of its order structure.
He found a coloring $c: {\[\Q\]}^{2} \rightarrow \nn$ such that for any $X \subseteq \Q$, if $X$ is homeomorphic to $\Q$, then for all $n \in \nn$, there exists $v \in {\[X\]}^{2}$ with $c(v) = n$.
In other words, he established the special case of Theorem \ref{thm:todorcevicweiss} saying that $\Q$ fails to have finite Ramsey degree in dimension $2$ within any countable metrizable space.
If a set of reals is homeomorphic to $\Q$, then it contains a subset which is order isomorphic to $\Q$, but the reserve is not true.
\emph{A priori}, this suggests that finding a homeomorphic copy of $\Q$ with some property is more difficult than finding an order isomorphic copy of $\Q$ with that property, and Baumgartner's result shows this is fundamentally more difficult in Ramsey theory.

It should also be noted that results of Shelah in \cite{sh:288} and \cite{sh:546} hinted at the truth of Theorem \ref{thm:metricspaces} for the space $\R$ because they established the consistency of a statement implying $\R \rightarrow {\left( \Q \right)}^{2}_{l, 2}$ for all $l < \omega.$
Assuming suitable large cardinals, Shelah constructed a model of set theory where for any natural number $l \geq 1$ and any coloring $c: {\[\R\]}^{2} \rightarrow l$, there is an uncountable set $X \subseteq \R$ such that $c$ uses at most $2$ colors on ${\[X\]}^{2}$.
It should be noted that Shelah's result is a consistency result, and not a direct implication.
In Shelah's model, the cardinality of $\R$ is quite large, for example it is a fixed point of the $\aleph$-operation, and there is not much control over the colorings of the pairs for any set of reals whose cardinality is smaller than that of $\R$.
Indeed, by a well-known theorem of Todorcevic~\cite{squarebracket}, if $X$ is any set of size ${\aleph}_{1}$, then there is a coloring of ${\[X\]}^{2}$ into ${\aleph}_{1}$ many colors so that every uncountable subset of $X$ contains a pair of every color.
Shelah's techniques do not provide information about the Ramsey degrees of $\Q$ in other topological spaces which do not contain a homeomorphic copy of $\R$.
\section{Ramsey degrees within a wider class of spaces}\label{intro2}
Several previous results in this general area of Ramsey theory for topological spaces had suggested that the $2$-dimensional Ramsey degree of $\Q$ should be $2$ within a much wider class of spaces.
These results concern the computation of the Ramsey degrees of a space much simpler than $\Q,$ namely the converging sequence, which is most naturally represented as the ordinal $\omega +1=\omega\cup\{\omega\}$ with its topology induced by the $\in$-ordering on ordinals.
For example, Baumgartner~\cite{baumtop} has shown that $X \not\rightarrow {\left( \omega+1 \right)}^{2}_{2}$ for every countable topological space $X$ and that on the other hand, $\Q \rightarrow {\left( \omega+1 \right)}^{2}_{l,2}$ for all $l<\omega.$
Thus, the space $\omega+1$ has Ramsey degree 2 in the class of all countable dense in itself metrizable spaces.
It turns out that trying to extend Baumgartner's computation of the Ramsey degrees of $\omega+1$ to an optimal class of spaces will also give us hints towards an optimal class of spaces where the Ramsey degree of $\Q$ is equal to 2.
For example, it is not difficult to show that if $X$ is any uncountable set of reals, then $X \rightarrow {\left( \omega+1 \right)}^{2}_{l}$ for all $l<\omega,$ i.e.\@, that the Ramsey degree of $\omega+1$ in the class of all uncountable sets of reals is equal to 1.
This was generalized in an unpublished note of the second author from 1996 (extending a previous result from \cite{todorcevicweissnotes}) as follows.
\begin{Def} \label{def:pointcountable}
 Let $\pr{X}{\TT}$ be a topological space.
 A base $\BB \subseteq \TT$ is said to be \emph{point-countable} if for each $x \in X$, $\{U \in \BB: x \in U\}$ is countable.
\end{Def} 
\begin{Theorem}[\cite{todorcevic96notes}] \label{omegaplusone}
 The following are equivalent for an arbitrary regular space $X$ with a point-countable base:
\begin{enumerate}
\item
there is no well-ordering of $X$ with all initial segments closed in $X;$
\item
$X \rightarrow {\left( \omega+1 \right)}^{2}_{2};$
\item
$X \rightarrow {\left( \omega+1 \right)}^{k}_{l}$ for all natural numbers $k,l\geq 1.$
\end{enumerate}
\end{Theorem}
It turns out that the negation of (1) of Theorem \ref{omegaplusone} is one of the standard smallness requirements on a space, which in the class of metrizable spaces, is equivalent to $\sigma$-discreteness.
Thus we have the following definition.
\begin{Def} \label{def:leftseparated}
 A topological space $\pr{X}{\TT}$ is said to be \emph{left-separated} if there exists a well-ordering $\wor$ of $X$ so that for each $x \in X$, $\{y \in X: y \: \wor \: x\}$ is a closed set.
\end{Def}
The proof of the implication from (2) to (1) in Theorem \ref{omegaplusone} has some information of interest to us here.
To see this assume that (1) fails and fix a well-ordering $\wor$ on $X$ with all initial segments closed.
So for every $y \in X$ we can fix a closed neighborhood ${U}_{y}$ of $y$ which is disjoint from $\{x \in X: x \: \wor \: y\}.$
Define $c:{[X]}^{2} \rightarrow 2$ by letting $c(x, y) = 0$ iff $y \in {U}_{x}$ for all pairs $x, y \in X$ satisfying $x \: \wor \: y$.
It is easily checked that subsets $Y$ of $X$ for which $c$ is constant on ${[Y]}^{2}$ must be discrete.
So in particular $X \not\rightarrow {\left( \omega +1 \right)}^{2}_{2}$, and therefore $X \not\rightarrow {(\Q)}^{2}_{2}$.
In \cite{GS}, Gerlits and Szentmikl{\' o}ssy have given an interesting variation of left separation which is equivalent to it in the class of spaces with a point countable base.
It is condition (1) of the following Corollary.
\begin{Cor} \label{cor:GS}
 The following are equivalent for every regular space $X$ with a point countable base:
 \begin{enumerate}
 \item
  there is a neighborhood assignment ${U}_{x} (x \in X)$ such that for all infinite $Y \subseteq X$ there is $y \in Y$ such that $\{x \in Y: y \notin {U}_{x}\}$ is infinite;
 \item
  there is a well-ordering of $X$ with all initial segments closed;
 \item
  $X \not\rightarrow {\left( \omega+1 \right)}^{2}_{3}.$
 \end{enumerate}
\end{Cor}
\begin{proof}
 The equivalence of (2) and (3) is by Theorem \ref{omegaplusone}.
 It is clear that (2) implies (1) using a neighborhood assignment such that ${U}_{y} \cap \{x \in X: x \: \wor \: y\} = \emptyset$ for all $y \in X,$ where $\wor$ is a well-ordering on $X$ with all initial segments closed.
 To show that (1) implies (3), consider the coloring $c:{[X]}^{2} \rightarrow 3$ defined as follows, where $\wor$ is a fixed well-ordering of $X$ and where we assume that the neighborhood assignment ${U}_{x} (x \in X)$ witnessing (1) consists of closed neighborhoods.
 For $x \: \wor \: y,$ set $c(x,y)=0$ if $x \notin {U}_{y}$ and $y \notin {U}_{x};$ set $c(x,y)=1$ if $x \in {U}_{y};$ finally, set $c(x,y)=2$ if $x \notin {U}_{y}$ but $y \in {U}_{x}.$
 Note that if $Y \subseteq X$ is such that $c''{[Y]}^{2} = \{1\}$, then $Y$ must be finite or else we would contradict (1).
 Note also that any $Y \subseteq X$ such that $c''{[Y]}^{2} = \{0\}$ or $c''{[Y]}^{2} = \{2\}$ must be discrete.
 So the coloring $c$ witnesses (3).
\end{proof}
We have already noted that one source of inspiration for this paper comes from Todorcevic's solution, through large cardinals, to a problem of Haydon.
A space $X$ is called \emph{universally meager} if every continuous function from a Baire space into $X$ must be constant on some non-meager subset of the Baire space.
We recall that the dual notion of a \emph{universally null} set is a well-studied notion, and especially its strengthening, the notion of a \emph{strong measure zero} set due to E.\@ Borel~\cite{borel}.
Recall that Borel~\cite{borel} conjectured that his notion coincides with the countability requirement for sets of reals, a conjecture which was proved to be consistent by Laver~\cite{laver} much later.
Thus, since the notion of universally meager is a strengthening of the direct dual of the notion of universally null, following the analogy, it is natural to conjecture that all universally meager sets of reals must be countable.
In fact, motivated by a problem about generic continuity and generic differentiability of functions on general Banach spaces, Haydon~\cite{haydon90} asked whether it is the case that a metrizable space is universally meager if and only if it is $\sigma$-scattered.
In \cite{genericcontinuity}, Todorcevic gave a positive answer to Haydon's question by showing that the existence of an uncountable strongly compact cardinal implies that if $X$ is any regular space with a point-countable base, then $X$ is universally meager if and only if it is left-separated.
All of the above mentioned results suggested the following project.
\begin{gp} \label{gp:generalproblem}
 Discover the optimal class of regular topological spaces in which the $2$-dimensional Ramsey degree of $\Q$ is at most $2$, and more generally, the $k$-dimensional Ramsey degree of $\Q$ is at most $k!(k-1)!$.
\end{gp}
 In this paper, we will address the general problem in dimension $2$ for all regular spaces with point-countable bases.
 Our main theorem is the following.
\begin{Theorem} \label{thm:general}
 Assume either that there is a proper class of Woodin cardinals or one uncountable strongly compact cardinal.
 If $X$ is any regular space that is not left-separated and has a point countable base, then the $2$-dimensional Ramsey degree of $\Q$ within $X$ is at most $2$. 
\end{Theorem}
Note that Theorem \ref{thm:metricspaces} immediately follows from Theorem \ref{thm:general} because metrizable spaces have point countable bases, and they are left-separated if and only if they are $\sigma$-discrete.

We will be treating higher dimensions and regular spaces without point-countable bases in forthcoming papers.
\section{Notation} \label{sec:notation}
Our set-theoretic notation is standard.
If $\lambda$ is an infinite cardinal, then $H(\lambda)$ denotes the set of all sets that are hereditarily of cardinality $< \lambda$.
The notation $M \prec H(\lambda)$ means that $\pr{M}{\in}$ is an elementary submodel of the structure $\pr{H(\lambda)}{\in}$.
For any $A$, $\Pset(A)$ denotes the powerset of $A$ -- that is, $\Pset(A) = \{a: a \subseteq A\}$.
For any $A$ and $B$, ${A}^{B}$ is the collection of all functions from $B$ to $A$.
If $\delta$ is an ordinal, then ${A}^{< \delta} = {\bigcup}_{\gamma < \delta}{{A}^{\gamma}}$.
If $f$ is a function, then $\dom(f)$ denotes the domain of $f$, and if $X \subseteq \dom(f)$, then $f''X$ is the image of $X$ under $f$ -- that is, $f''X = \{f(x): x \in X\}$.   
\section{Some preliminaries} \label{sec:prelim}
Properties of stationary sets will be used extensively in the proof of the main result.
In this section, we will collect together important facts needed in Section \ref{sec:main}.
Most of this material is standard.
We will need to deal only with stationary subsets of ${\[A\]}^{< {\aleph}_{1}}$, for various sets $A$.
Other more general notions of stationarity have been considered in the literature.
For example, one could talk about stationary subsets of $\Pset(A)$, for any non-empty set $A$.
The interested reader may consult \cite{paultower} or \cite{Je}.   
\begin{Def} \label{def:stat}
  Let $A$ be a non-empty set.
  $C \subseteq {\[A\]}^{< {\aleph}_{1}}$ is called a \emph{club in ${\[A\]}^{< {\aleph}_{1}}$} if the following two things hold:
  \begin{enumerate}
    \item
    for any $N \in {\[A\]}^{< {\aleph}_{1}}$, there exists $M \in C$ with $N \subseteq M$;
    \item
    for any $0 < \xi < {\aleph}_{1}$ and for any sequence $\seq{M}{\zeta}{<}{\xi}$ of elements of $C$, if $\forall \zeta' \leq \zeta < \xi\[{M}_{\zeta'} \subseteq {M}_{\zeta}\]$, then ${\bigcup}_{\zeta < \xi}{{M}_{\zeta}} \in C$.
  \end{enumerate}
  We say that $S \subseteq {\[A\]}^{< {\aleph}_{1}}$ is \emph{stationary in ${\[A\]}^{< {\aleph}_{1}}$} if $S \cap C \neq 0$, for every $C \subseteq {\[A\]}^{< {\aleph}_{1}}$ which is a club in ${\[A\]}^{< {\aleph}_{1}}$.
  And $S \subseteq {\[A\]}^{< {\aleph}_{1}}$ is said to be \emph{non-stationary in ${\[A\]}^{< {\aleph}_{1}}$} if it is not stationary in ${\[A\]}^{< {\aleph}_{1}}$.
\end{Def}
One of the salient facts about the non-stationary subsets of ${\[A\]}^{< {\aleph}_{1}}$ is that they form a normal $\sigma$-ideal.
\begin{Theorem}[Jech~\cite{Je}]\label{thm:normality}
 Let $A$ be a non-empty set.
 If $\F$ is any countable family of non-stationary subsets of ${\[A\]}^{< {\aleph}_{1}}$, then $\bigcup \F$ is also a non-stationary subset of ${\[A\]}^{< {\aleph}_{1}}$.
 If $S$ is a stationary subset of ${\[A\]}^{< {\aleph}_{1}}$, and $F$ is a function such that $\dom(F) = S$ and $\forall M \in S\[F(M) \in M\]$, then there exists an $m$ so that $\{M \in S: F(M) = m\}$ is a stationary subset of ${\[A\]}^{< {\aleph}_{1}}$.
\end{Theorem}
The last statement of Theorem \ref{thm:normality} is usually called \emph{the pressing down lemma}.
The following theorem is a well-known fact about clubs and stationary sets.
It governs the behavior of clubs and stationary sets under projections and pullbacks. 
The reader may refer to Kanamori~\cite{kanamori} or to Jech~\cite{Je} for a proof.
This theorem below is true even when ${\aleph}_{1}$ is replaced with an arbitrary regular uncountable cardinal.
It is also true for the more general notion of club and stationary set in $\Pset(X)$.
The proof of a version that is applicable to the more general notion of club and stationary set may be found in Larson~\cite{paultower}.
\begin{Theorem} \label{thm:statprojection}
  Let $X$ and $Y$ be non-empty sets with $X \subseteq Y$.
  Then the following hold:
  \begin{enumerate}
    \item
     if $C \subseteq {\[X\]}^{< {\aleph}_{1}}$ is a club in ${\[X\]}^{< {\aleph}_{1}}$, then 
     \begin{align*}
      {C}^{\uparrow Y} = \left\{ M \in {\[Y\]}^{< {\aleph}_{1}}: M \cap X \in C \right\}
     \end{align*}
     is a club in ${\[Y\]}^{< {\aleph}_{1}}$;
    \item
    if $S \subseteq {\[Y\]}^{< {\aleph}_{1}}$ is stationary in ${\[Y\]}^{< {\aleph}_{1}}$, then
    \begin{align*}
      {S}_{\downarrow X} = \left\{ M \cap X: M \in S \right\} 
    \end{align*}
    is stationary in ${\[X\]}^{< {\aleph}_{1}}$;
    \item
    if $C \subseteq {\[Y\]}^{< {\aleph}_{1}}$ is a club in ${\[Y\]}^{< {\aleph}_{1}}$, then
    \begin{align*}
      {C}_{\downarrow X} = \left\{ M \cap X: M \in C \right\} \subseteq {\[X\]}^{< {\aleph}_{1}}
    \end{align*}
    and ${C}_{\downarrow X}$ contains a club in ${\[X\]}^{< {\aleph}_{1}}$;
    \item
    if $S \subseteq {\[X\]}^{< {\aleph}_{1}}$ is stationary in ${\[X\]}^{< {\aleph}_{1}}$, then
    \begin{align*}
     {S}^{\uparrow Y} = \left\{ M \in {\[Y\]}^{< {\aleph}_{1}}: M \cap X \in S \right\} 
    \end{align*}
    is stationary in ${\[Y\]}^{< {\aleph}_{1}}$.
  \end{enumerate}
\end{Theorem}
Note the asymmetry between (1) and (3), and the symmetry between (2) and (4).
We will really only make use of (2) and (4).
The relevance of stationary sets to left-separation of topological spaces is taken up next.
\begin{Def} \label{def:closureBY}
  Let $\pr{X}{\TT}$ be a topological space.
  For any $A \subseteq X$, $\overline{A}$ will denote the closure of $A$.
  Given a base $\BB \subseteq \TT$ and a $Y \subseteq X$, ${\BB}_{Y}$ will denote $\left\{ U \in \BB: U \cap Y \neq 0 \right\}$.
\end{Def}
Theorem \ref{thm:fleissner} is a deep characterization of regular left-separated spaces having a point-countable base in terms of non-stationarity of the collection of all countable closed subsets of the space.
It first appears in Fleissner~\cite{fle}.
Indeed the theorem is also valid for ${T}_{1}$ spaces.
However, all of our spaces are assumed to be regular because we would like to be able to find subspaces homeomorphic to $\Q$ within them.  
\begin{Theorem}[see \cite{fle}]\label{thm:fleissner}
  If $\pr{X}{\TT}$ is a regular space which has a point-countable base, then $\pr{X}{\TT}$ is not left-separated if and only if $\left\{ N \in {\[X\]}^{< {\aleph}_{1}}: \overline{N} \setminus N \neq 0 \right\}$ is stationary in ${\[X\]}^{< {\aleph}_{1}}$.
\end{Theorem}
As mentioned in the introduction, metrizable spaces that are not $\sigma$-discrete are one class of examples of regular non-left-separated spaces with point-countable bases.
Another example is a special stationary Aronszajn line.
One of the benefits of a point-countable base is that any countable set which is sufficiently closed under definable operations must contain all the members of the base around any point in its closure.
This fact is proved in the next lemma, which will enable us to apply the pressing down lemma.   
\begin{Lemma} \label{lem:elementary1}
  Let $\pr{X}{\TT}$ be a topological space with a point-countable base $\BB \subseteq \TT$.
  Let $\chi$ be any uncountable regular cardinal and suppose that $M \prec H(\chi)$ with $\lc M \rc = {\aleph}_{0}$ and $\pr{X}{\TT}, \BB \in M$.
  If $x \in \overline{X \cap M} \setminus M$, then ${\BB}_{\{x\}} \subseteq M$.
\end{Lemma}
\begin{proof}
  Consider any $U \in {\BB}_{\{x\}}$.
  $U$ is an open set with $x \in U$, and so $U \cap X \cap M \neq 0$.
  Choose $y \in U \cap M$.
  Thus $\{y\} \in M$ and ${\BB}_{\{y\}} \in M$.
  Since $\BB$ is point-countable, ${\BB}_{\{y\}}$ is a countable set.
  Therefore ${\BB}_{\{y\}} \subseteq M$.
  As $U \in {\BB}_{\{y\}}$, $U \in M$.
  This shows ${\BB}_{\{x\}} \subseteq M$, as needed. 
\end{proof}
The countable stationary tower will be our main tool for proving Theorem~\ref{thm:general}.
Building on the groundbreaking work of Foreman, Magidor, and Shelah~\cite{fms1}, Woodin introduced the stationary tower in \cite{woodin88} and established a wide variety of results in set theory with it.
Larson~\cite{paultower} provides an excellent and accessible introduction to the stationary tower and its applications.
A more advanced reference is Woodin~\cite{woodin}.
Towers of ideals, including several variants of the stationary tower, and their associated generic elementary embeddings are studied in Foreman~\cite{foremanhandbook}. 
Kanamori~\cite{kanamori} provides an introduction to large cardinals. 
\begin{Def} \label{def:countabletower}
  Let $\delta$ be a strongly inaccessible cardinal.
  As usual, ${V}_{\delta}$ denotes $\{a: \rk(a) < \delta\}$.
  The \emph{countable stationary tower up to $\delta$}, denoted ${\Q}_{< \delta}$, is defined to be the collection of all $\pr{A}{S} \in {V}_{\delta}$ such that $A$ is a non-empty set and $S \subseteq {\[A\]}^{< {\aleph}_{1}}$ is stationary in ${\[A\]}^{< {\aleph}_{1}}$.
  Elements of ${\Q}_{< \delta}$ will sometimes be called \emph{conditions in ${\Q}_{< \delta}$}, or simply \emph{conditions}.

  An ordering on ${\Q}_{< \delta}$ is defined as follows.
  For $\pr{A}{S}, \pr{B}{T} \in {\Q}_{< \delta}$, define $\pr{B}{T} \leq \pr{A}{S}$ to mean that $B \supseteq A$ and $T \subseteq {S}^{\uparrow B}$.
  It is easily checked that $\leq$ is a partial order on ${\Q}_{< \delta}$.
  Observe also that for any $\pr{B}{T}, \pr{A}{S} \in {\Q}_{< \delta}$, $\pr{B}{T} \leq \pr{A}{S}$ if and only if $B \supseteq A$ and ${T}_{\downarrow A} \subseteq S$.
  
  If $p \in {\Q}_{< \delta}$ and $D \subseteq {\Q}_{< \delta}$, then $D$ is said to be \emph{dense below} $p$ if for each $\pr{A}{S} \leq p$, there exists $\pr{B}{T} \in D$ with $\pr{B}{T} \leq \pr{A}{S}$.
\end{Def}
Fix a strongly inaccessible cardinal $\delta > \omega$ for the remainder of this section.
The following lemma will be useful in conjunction with Lemma \ref{lem:elementary1} and the pressing down lemma.
\begin{Lemma} \label{lem:gooddense}
  Let $\pr{X}{\TT, \BB} \in {V}_{\delta}$ be a regular topological space where $\BB \subseteq \TT$ is a point-countable base.
  If $\pr{X}{\TT}$ is not left-separated, then
  \begin{align*}
   p = \left\langle X, \left\{ N \in {\[X\]}^{< {\aleph}_{1}}: \overline{N}\setminus N \neq 0 \right\} \right\rangle \in {\Q}_{< \delta}.
  \end{align*}
  Moreover, the collection of all $\pr{B}{T} \leq p$ with the property that there exists an uncountable regular cardinal $\chi$ such that $B = H(\chi)$ and
  \begin{align*}
   \forall M \in T\[\lc M\rc = {\aleph}_{0} \ \text{and} \ M \prec H(\chi) \ \text{and} \ \pr{X}{\TT}, \BB \in M \ \text{and} \ \overline{X \cap M} \setminus M \neq 0\]
  \end{align*}
  is dense below $p$.
\end{Lemma}
\begin{proof}
  The hypotheses together with Theorem \ref{thm:fleissner} imply that $X$ is a non-empty set, $p \in {V}_{\delta}$, and that $p$ is a condition in ${\Q}_{< \delta}$.
  For the second part, let $\pr{A}{S} \leq p$.
  Fix an uncountable regular cardinal $\chi$ with $\left\{A, \pr{X}{\TT}, \BB \right\} \subseteq H(\chi)$ and $H(\chi) \in {V}_{\delta}$.
  Let $B = H(\chi)$.
  Since $A \subseteq B$, $\pr{B}{{S}^{\uparrow B}} \leq \pr{A}{S}$.
  Now it is well-known that $C = \left\{M \in {\[B\]}^{< {\aleph}_{1}}: M \prec H(\chi) \ \text{and} \ \{ A, \pr{X}{\TT}, \BB \} \subseteq M \right\}$ is a club in ${\[B\]}^{< {\aleph}_{1}}$.
  Let $T = C \cap {S}^{\uparrow B}$.
  Then $\pr{B}{T} \leq \pr{B}{{S}^{\uparrow B}} \leq \pr{A}{S} \leq p$ and it is as required.
\end{proof}
Todorcevic~\cite{genericcontinuity} defines an ideal ${\I}_{{\omega}_{1}}(\delta)$ as follows.
Let us say that a set $T \subseteq {\[\delta\]}^{< {\aleph}_{1}}$ \emph{depends on a bounded set of coordinates in $\delta$}, if there exists a bounded subset $A \subseteq \delta$ with the property that for all $M, M' \in {\[\delta\]}^{< {\aleph}_{1}}$, if $M \cap A = M' \cap A$, then $M \in T$ if and only if $M' \in T$.
${\F}_{{\omega}_{1}}(\delta)$ denotes the collection of all $T \subseteq {\[\delta\]}^{< {\aleph}_{1}}$ that depend on a bounded set of coordinates in $\delta$.
For a bounded subset $A \subseteq \delta$ and a function $f: {A}^{< \omega} \rightarrow A$, ${C}_{f}$ denotes $\left\{M \in {\[\delta\]}^{< {\aleph}_{1}}: f''\left( {\left( M \cap A \right)}^{< \omega} \right) \subseteq M \right\}$.
${\I}_{{\omega}_{1}}(\delta)$ is the collection of all $T \subseteq {\[\delta\]}^{< {\aleph}_{1}}$ for which there exist a bounded $A \subseteq \delta$ and a function $f: {A}^{< \omega} \rightarrow A$ such that $T \cap {C}_{f} = \emptyset$.
Finally, ${\BB}_{{\omega}_{1}}(\delta) = {\F}_{{\omega}_{1}}(\delta) \setminus {\I}_{{\omega}_{1}}(\delta)$.
While we will not be working with any of these collections directly, it is worth noting that there is a natural one-to-one correspondence between the members of ${\Q}_{< \delta}$ and ${\BB}_{{\omega}_{1}}(\delta)$.

We now consider a version of the Banach-Mazur game played with conditions in ${\Q}_{< \delta}$.
It is also similar to the precipitous game (see \cite{Je}).
\begin{Def} \label{def:game1}
  Define a two-player game $\Game(\delta)$ as follows.
  Two players Empty and Non-Empty take turns playing conditions in ${\Q}_{< \delta}$, with Empty making the first move.
  When one of the players has played $\pr{{A}_{n}}{{S}_{n}} \in {\Q}_{< \delta}$, his opponent is required to play $\pr{{A}_{n + 1}}{{S}_{n + 1}} \leq \pr{{A}_{n}}{{S}_{n}}$.
  Thus each run of the game produces a sequence
  \begin{align*}
    \begin{tabular}{c| c c c c c c}
      $\textrm{Empty}$& $\pr{{A}_{0}}{{S}_{0}}$ &  & $\pr{{A}_{2}}{{S}_{2}}$& & $\dotsb$ & \\
      \hline
      $\textrm{Non-Empty}$& & $\pr{{A}_{1}}{{S}_{1}}$& &$\dotsb$ &
    \end{tabular}
  \end{align*}
  such that for each $n \in \omega$, $\pr{{A}_{2n}}{{S}_{2n}}$ has been played by Empty, $\pr{{A}_{2n+1}}{{S}_{2n+1}}$ has been played by Non-Empty and $\pr{{A}_{n+1}}{{S}_{n+1}} \leq \pr{{A}_{n}}{{S}_{n}}$.
  Non-Empty wins this particular run of $\Game(\delta)$ if and only if there exists a sequence $\seq{N}{l}{\in}{\omega}$ such that $\forall l \in \omega\[{N}_{l} \in {S}_{l}\]$ and $\forall k \leq l\[{N}_{k} = {N}_{l} \cap {A}_{k}\]$.
\end{Def}
The following important theorem tells us that if $\delta$ is a suitable large cardinal, then the Empty player does not have a winning strategy in $\Game(\delta)$.
It is essentially equivalent to the well-known fact that the generic ultrapower of the universe induced by ${\Q}_{< \delta}$ is closed under $\omega$-sequences in the generic extension by ${\Q}_{< \delta}$.
A version of this theorem for the collection ${\BB}_{{\omega}_{1}}(\delta)$ was proved by Todorcevic in \cite{genericcontinuity}.
In fact, Theorem \ref{thm:banachmazur} also follows from the proof of Lemma 2.3 from Todorcevic~\cite{genericcontinuity} via the correspondence between elements of ${\Q}_{< \delta}$ and ${\BB}_{{\omega}_{1}}(\delta)$ discussed earlier.
Alternatively, the proof of Lemma 2.5.6 from Larson~\cite{paultower} can be easily adapted to prove Theorem \ref{thm:banachmazur}. 
\begin{Theorem} \label{thm:banachmazur}
  If $\delta$ is a Woodin cardinal or an uncountable strongly compact cardinal, then Empty does not have a winning strategy in $\Game(\delta)$.
\end{Theorem}
\section{Main Theorem} \label{sec:main}
Fix, once and for all, an uncountable cardinal $\delta$, which is either Woodin or strongly compact.
Fix in addition a regular topological space $\pr{X}{\TT, \BB} \in {V}_{\delta}$, where $\BB \subseteq \TT$ is a point-countable base and $\pr{X}{\TT}$ is not left-separated.
Put ${A}_{0} = X$ and ${S}_{0} = \left\{ N \in {\[X\]}^{< {\aleph}_{1}}: \overline{N}\setminus N \neq 0 \right\}$.
Note that $\pr{{A}_{0}}{{S}_{0}} \in {\Q}_{< \delta}$.
\begin{Def} \label{def:function}
  Fix a function $F: {S}_{0} \rightarrow X$ such that $F(N) \in \overline{N} \setminus N$, for each $N \in {S}_{0}$.
  If $\pr{A}{S} \leq \pr{{A}_{0}}{{S}_{0}}$, then for any $M \in S$, $M \cap {A}_{0} \in {S}_{0}$, and we will abuse notation and write $F(M)$ to mean $F(M \cap {A}_{0})$. 
\end{Def}
We will first prove a sequence of simple lemmas establishing some useful properties of $F$ and of the neighborhoods in $\BB$.
The first property is that $F$ is ``nowhere constant'', meaning that the preimage of every point in ${A}_{0}$ is non-stationary.
\begin{Lemma} \label{lem:nowhere}
  For any $\pr{A}{S} \leq \pr{{A}_{0}}{{S}_{0}}$ and any $x \in {A}_{0}$, $\{M \in S: F(M) = x \}$ is non-stationary in ${\[A\]}^{< {\aleph}_{1}}$.
\end{Lemma}
\begin{proof}
  Suppose for a contradiction that $S' = \{M \in S: F(M) = x \}$ is stationary in ${\[A\]}^{< {\aleph}_{1}}$.
  Note that $C = \{M \in {\[A\]}^{< {\aleph}_{1}}: \{x\} \subseteq M\}$ is a club in ${\[A\]}^{< {\aleph}_{1}}$.
  Choose $M \in S$ with $x \in M$ and $F(M) = x$.
  Put $N = M \cap {A}_{0}$.
  Then $x \in N$.
  However $x = F(M) = F(N) \in \overline{N}\setminus N$.
  This is a contradiction completing the proof.
\end{proof}
The next property concerns the ``largeness'' of neighborhoods of points in $X$.
For any condition in ${\Q}_{< \delta}$ below $\pr{{A}_{0}}{{S}_{0}}$, every neighborhood of almost every point in the image of that condition has large intersection with the same image.
This is proved by a simple application of the pressing-down lemma.
\begin{Def} \label{def:Bn}
  For each $x \in X$, fix an enumeration $\left\langle {U}_{x, n}: n \in \omega \right\rangle$ of the set $\left\{ U \in \BB: x \in U \right\}$.
  For any $\pr{A}{S} \leq \pr{{A}_{0}}{{S}_{0}}$, we will say that $M \in S$ is \emph{bad} if there exists $n \in \omega$ such that $\left\{M' \in S: F(M') \in {U}_{F(M), n}\right\}$ is non-stationary in ${\[A\]}^{< {\aleph}_{1}}$.
\end{Def}
\begin{Lemma} \label{lem:badsmall}
  Suppose $\pr{A}{S} \leq \pr{{A}_{0}}{{S}_{0}}$.
  Then $\{M \in S: M \ \text{is bad}\}$ is non-stationary.
\end{Lemma}
\begin{proof}
Write ${S}_{1} = \{M \in S: M \ \text{is bad}\}$.
Assume for a contradiction that ${S}_{1}$ is stationary in ${\[A\]}^{< {\aleph}_{1}}$.
Then $\pr{A}{{S}_{1}} \leq \pr{A}{S} \leq \pr{{A}_{0}}{{S}_{0}}$.
Applying Lemma \ref{lem:gooddense}, there exists $\pr{B}{T} \leq \pr{A}{{S}_{1}}$ with the property that $B = H(\chi)$ where $\chi$ is an uncountable regular cardinal, and for all $K \in T$, $K \prec H(\chi)$, $\lc K \rc = {\aleph}_{0}$ and $\pr{X}{\TT}, \BB \in K$.
For any $K \in T$, $M = K \cap A \in {S}_{1}$, and so $M$ is bad, which means that there exists $n \in \omega$ so that $\left\{ M' \in S: F(M') \in {U}_{F(M), n} \right\}$ is non-stationary in ${\[A\]}^{< {\aleph}_{1}}$.
Note that $F(K) = F(M)$ and that ${U}_{F(K), n} \in K$ because of Lemma \ref{lem:elementary1}.
Thus for each $K \in T$, we have ${U}_{K} \in K$ such that $F(K) \in {U}_{K}$ and $\left\{M' \in S: F(M') \in {U}_{K} \right\}$ is non-stationary in ${\[A\]}^{< {\aleph}_{1}}$.
By the pressing down lemma, there exists $U$ so that
\begin{align*}
 T' = \{K \in T: {U}_{K} = U \}
\end{align*}
is stationary in ${\[B\]}^{< {\aleph}_{1}}$.
Since $T' \neq \emptyset$, $\left\{ M' \in S: F(M') \in U \right\}$ is non-stationary in ${\[A\]}^{< {\aleph}_{1}}$.
On the other hand, ${T}_{\downarrow A}'$ is stationary in ${\[A\]}^{< {\aleph}_{1}}$ and
\begin{align*}
 {T}_{\downarrow A}' \subseteq \left\{ M' \in S: F(M') \in U \right\}.
\end{align*}
This is a contradiction that concludes the proof.
\end{proof}
Lemma \ref{lem:badsmall} says that for any $\pr{A}{S} \leq \pr{{A}_{0}}{{S}_{0}}$, the set $\{M \in S: M \ \text{is not bad}\}$ is almost equal to $S$.
Therefore once all the bad points in $S$ have been thrown away, none of the remaining points can be bad in what's left.
So there is no need to repeat the operation of throwing away bad points.
This is what Lemma \ref{lem:stildeisgood} says.
\begin{Def} \label{def:stilde}
  For any $\pr{A}{S} \leq \pr{{A}_{0}}{{S}_{0}}$, define $\widetilde{S} = \{M \in S: M \ \text{is not bad}\}$.
  By Lemma \ref{lem:badsmall}, $\pr{A}{\widetilde{S}} \in {\Q}_{< \delta}$ and $\pr{A}{\widetilde{S}} \leq \pr{A}{S}$.
\end{Def}
\begin{Lemma} \label{lem:stildeisgood}
  Let $\pr{A}{S} \leq \pr{{A}_{0}}{{S}_{0}}$.
  For any $M \in \widetilde{S}$ and any $n \in \omega$, 
  \begin{align*}
   \left\{ M' \in \widetilde{S}: F(M') \in {U}_{F(M), n} \right\}
  \end{align*} 
  is stationary in ${\[A\]}^{< {\aleph}_{1}}$.
\end{Lemma}
\begin{proof}
  Take any $M \in \widetilde{S}$ and any $n \in \omega$.
  Then $M$ is not bad in $S$, which means that $\left\{ M' \in S: F(M') \in {U}_{F(M), n} \right\}$ is stationary in ${\[A\]}^{< {\aleph}_{1}}$.
  Since $S \setminus \widetilde{S}$ is non-stationary in ${\[A\]}^{< {\aleph}_{1}}$, it follows that $\widetilde{S} \cap \left\{ M' \in S: F(M') \in {U}_{F(M), n} \right\} = \left\{ M' \in \widetilde{S}: F(M') \in {U}_{F(M), n} \right\}$ is stationary in ${\[A\]}^{< {\aleph}_{1}}$.
\end{proof}
\begin{Def} \label{def:coloring}
  Fix $l \in \omega$ with $l > 0$ and fix $c: {\[X\]}^{2} \rightarrow l$.
  Suppose $\pr{A}{S} \leq \pr{{A}_{0}}{{S}_{0}}$ and $\pr{B}{T} \leq \pr{{A}_{0}}{{S}_{0}}$.
  For any $i \in l$ and any $M \in S$, define $\KK(c, i, M, B, T)$ to be $\left\{ M' \in T: F(M) \neq F(M') \ \text{and} \ c(F(M), F(M')) = i \right\}$.
  We will say that $M$ is \emph{$i$-large in $\pr{B}{T}$ w.r.t.\@ $c$} if $\KK(c, i, M, B, T)$ is stationary in ${\[B\]}^{< {\aleph}_{1}}$.
  
  For any $i, j \in l$, the pair $\pr{\pr{A}{S}}{\pr{B}{T}}$ is said to be \emph{$\pr{i}{j}$-saturated w.r.t.\@ $c$} if for any $\pr{A'}{S'} \leq \pr{A}{S}$ and any $\pr{B'}{T'} \leq \pr{B}{T}$, both of the clauses below hold:
  \begin{enumerate}
   \item
   $\left\{ M \in S': M \ \text{is} \ i\text{-large in} \ \pr{B'}{T'} \ \text{w.r.t.\@} \ c \right\}$ is stationary in ${\[A'\]}^{< {\aleph}_{1}}$; 
   \item
   $\left\{ K \in T': K \ \text{is} \ j\text{-large in} \ \pr{A'}{S'} \ \text{w.r.t.\@} \ c \right\}$ is stationary in ${\[B'\]}^{< {\aleph}_{1}}$.
 \end{enumerate}  
\end{Def}
Intuitively, if a pair $\pr{\pr{A}{S}}{\pr{B}{T}}$ is $\pr{i}{j}$-saturated w.r.t.\@ $c$, then the colors $i$ and $j$ occur in every rectangle whose sides are conditions below $\pr{A}{S}$ and $\pr{B}{T}$ in ${\Q}_{< \delta}$.
More precisely, any rectangle whose base is a condition below $\pr{A}{S}$ and whose height is a condition below $\pr{B}{T}$, must contain many vertical columns with a large collection of $i$-colored points, and also many horizontal rows with a large collection of $j$-colored points. 
\begin{Lemma} \label{lem:largenessprojects}
  Suppose $\pr{A'}{S'} \leq \pr{A}{S} \leq \pr{{A}_{0}}{{S}_{0}}$ and that $\pr{B'}{T'} \leq \pr{B}{T} \leq \pr{{A}_{0}}{{S}_{0}}$.
  For any $i \in l$ and any $M \in S'$, $\KK(c, i, M, B, T) = \KK(c, i, M \cap A, B, T)$.
  Also if $M$ is $i$-large in $\pr{B'}{T'}$ w.r.t.\@ $c$, then $M$ is $i$-large in $\pr{B}{T}$ w.r.t.\@ $c$.
  Further, if $M$ is $i$-large in $\pr{B'}{T'}$ w.r.t.\@ $c$, then $M \cap A$ is $i$-large in $\pr{B}{T}$ w.r.t.\@ $c$.
\end{Lemma}
\begin{proof}
Indeed $M \cap A \in S$ and $F(M) = F(M \cap A)$ and so
\begin{align*}
 \KK(c, i, M, B, T) &= \left\{ M' \in T: F(M \cap A) \neq F(M') \ \text{and} \ c(F(M \cap A), F(M')) = i \right\}\\
 &= \KK(c, i, M \cap A, B, T).
\end{align*}
Moreover if $M$ is $i$-large in $\pr{B'}{T'}$ w.r.t.\@ $c$, then $\KK(c, i, M, B', T')$ is stationary in ${\[B'\]}^{< {\aleph}_{1}}$.
Since ${\left( \KK(c, i, M, B', T') \right)}_{\downarrow B}$ is stationary in ${\[B\]}^{< {\aleph}_{1}}$ and since we have that ${\left( \KK(c, i, M, B', T') \right)}_{\downarrow B} \subseteq \KK(c, i, M, B, T) \subseteq T \subseteq {\[B\]}^{< {\aleph}_{1}}$, it follows that $\KK(c, i, M, B, T)$ is stationary in ${\[B\]}^{< {\aleph}_{1}}$.
Hence $M$ is $i$-large in $\pr{B}{T}$ w.r.t.\@ $c$.
Finally if $M$ is $i$-large in $\pr{B'}{T'}$ w.r.t.\@ $c$, then by what we have remarked up to now, $\KK(c, i, M, B, T) = \KK(c, i, M \cap A, B, T)$ is stationary in ${\[B\]}^{< {\aleph}_{1}}$, and so $M \cap A$ is $i$-large in $\pr{B}{T}$ w.r.t.\@ $c$.  
\end{proof}
The next lemma expresses the simple fact that for a fixed row or column in any rectangle, there must be a color that occurs frequently along that row or column. 
\begin{Lemma} \label{lem:oneislarge}
Suppose $\pr{A}{S} \leq \pr{{A}_{0}}{{S}_{0}}$ and $\pr{B}{T} \leq \pr{{A}_{0}}{{S}_{0}}$.
For each $M \in S$ and for each $K \in T$, there exists $\pr{i}{j} \in l \times l$ such that $M$ is $i$-large in $\pr{B}{T}$ w.r.t.\@ $c$ and $K$ is $j$-large in $\pr{A}{S}$ w.r.t.\@ $c$.
\end{Lemma}
\begin{proof}
  Put $x = F(M)$ and $y = F(K)$.
  By Lemma \ref{lem:nowhere}, $T' = \{K' \in T: F(K') \neq x\}$ is stationary in ${\[B\]}^{< {\aleph}_{1}}$ and $S' = \{M' \in S: F(M') \neq y\}$ is stationary in ${\[A\]}^{< {\aleph}_{1}}$.
  For each $i < l$, let ${T}_{i}' = \{K' \in T': c(F(M), F(K')) = i\}$ and let ${S}_{i}' = \{M' \in S': c(F(K), F(M')) = i\}$.
  There must be a pair $\pr{i}{j} \in l \times l$ such that ${T}_{i}' \subseteq {\[B\]}^{< {\aleph}_{1}}$ is stationary and ${S}_{j}' \subseteq {\[A\]}^{< {\aleph}_{1}}$ is stationary because ${\bigcup}_{i < l}{{T}_{i}'} = T'$ and ${\bigcup}_{i < l}{{S}_{i}'} = S'$.
  Since ${T}_{i}' \subseteq \KK(c, i, M, B, T) \subseteq {\[B\]}^{< {\aleph}_{1}}$ and ${S}_{j}' \subseteq \KK(c, j, K, A, S) \subseteq {\[A\]}^{< {\aleph}_{1}}$, $M$ is $i$-large in $\pr{B}{T}$ w.r.t.\@ $c$ and $K$ is $j$-large in $\pr{A}{S}$ w.r.t.\@ $c$. 
\end{proof}
It is obvious from the definition that the property of being $\pr{i}{j}$-saturated is hereditary.
We state this below as a separate fact because it will be useful, but we will omit the trivial proof.
\begin{Lemma} \label{lem:hereditary}
  Suppose $\pr{A'}{S'} \leq \pr{A}{S} \leq \pr{{A}_{0}}{{S}_{0}}$ and $\pr{B'}{T'} \leq \pr{B}{T} \leq \pr{{A}_{0}}{{S}_{0}}$.
  For any $\pr{i}{j} \in l \times l$, if $\pr{\pr{A}{S}}{\pr{B}{T}}$ is $\pr{i}{j}$-saturated, then so is $\pr{\pr{A'}{S'}}{\pr{B'}{T'}}$.
\end{Lemma}
The next lemma will play an important role in the rest of the proof.
It asserts the existence of a single pair of colors $\pr{i}{j}$ and a condition in ${\Q}_{< \delta}$ with the property that every condition below it in ${\Q}_{< \delta}$ can be split into an $\pr{i}{j}$-saturated pair.
The proof is an exhaustion argument. 
\begin{Lemma} \label{lem:existssaturated}
  There exist $\pr{i}{j} \in l \times l$ and $\pr{{A}_{1}}{{S}_{1}} \leq \pr{{A}_{0}}{{S}_{0}}$ such that for any $\pr{{A}_{2}}{{S}_{2}} \leq \pr{{A}_{1}}{{S}_{1}}$, there exist $\pr{A}{S} \leq \pr{{A}_{2}}{{S}_{2}}$ and $\pr{B}{T} \leq \pr{{A}_{2}}{{S}_{2}}$ such that $\pr{\pr{A}{S}}{\pr{B}{T}}$ is $\pr{i}{j}$-saturated w.r.t.\@ $c$.
\end{Lemma}
\begin{proof}
Since $l > 0$, we can enumerate the members of $l \times l$ as $\left\{ \pr{{i}_{1}}{{j}_{1}}, \dotsc, \pr{{i}_{{l}^{2}}}{{j}_{{l}^{2}}}\right\}$.
Suppose that the statement of the lemma fails.
Then there exists a sequence $\pr{{A}_{0}}{{S}_{0}} \geq \pr{{A}_{1}}{{S}_{1}} \geq \dotsb \geq \pr{{A}_{{l}^{2}}}{{S}_{{l}^{2}}}$ such that for each $1 \leq k \leq {l}^{2}$, $\pr{{A}_{k}}{{S}_{k}}$ has the property that for any $\pr{A}{S} \leq \pr{{A}_{k}}{{S}_{k}}$ and for any $\pr{B}{T} \leq \pr{{A}_{k}}{{S}_{k}}$, $\pr{\pr{A}{S}}{\pr{B}{T}}$ is not $\pr{{i}_{k}}{{j}_{k}}$-saturated w.r.t.\@ $c$.
Next build three sequences
\begin{align*}
 &\pr{{A}_{{l}^{2}}}{{S}_{{l}^{2}}} = \pr{{A}_{0}'}{{S}_{0}'} \geq \pr{{A}_{1}'}{{S}_{1}'} \geq \dotsb \geq \pr{{A}_{{l}^{2}}'}{{S}_{{l}^{2}}'}, \\
 &\pr{{A}_{{l}^{2}}}{{S}_{{l}^{2}}} = \pr{{B}_{0}'}{{T}_{0}'} \geq \pr{{B}_{1}'}{{T}_{1}'} \geq \dotsb \geq \pr{{B}_{{l}^{2}}'}{{T}_{{l}^{2}}'}\text{, and}\\
 &\pr{{S}^{\ast}_{1}}{{T}^{\ast}_{1}}, \dotsc, \pr{{S}^{\ast}_{{l}^{2}}}{{T}^{\ast}_{{l}^{2}}} \ \text{such that:}
\end{align*} 
\begin{enumerate}
  \item
  for each $1 \leq k \leq {l}^{2}$, ${S}^{\ast}_{k} \subseteq {S}_{k}'$ is non-stationary in ${\[{A}_{k}'\]}^{< {\aleph}_{1}}$ and ${T}^{\ast}_{k} \subseteq {T}_{k}'$ is non-stationary in ${\[{B}_{k}'\]}^{< {\aleph}_{1}}$.
  \item
  for each $1 \leq  k \leq {l}^{2}$,
 \begin{align*}
  &\text{\emph{either}} \ {S}^{\ast}_{k} = \left\{ M \in {S}_{k}': M \ \text{is} \ {i}_{k}\text{-large in} \ \pr{{B}_{k}'}{{T}_{k}'} \ \text{w.r.t.\@} \ c \right\}\\
  &\text{\emph{or}} \ {T}^{\ast}_{k} = \left\{ K \in {T}_{k}': K \ \text{is} \ {j}_{k}\text{-large in} \ \pr{{A}_{k}'}{{S}_{k}'} \ \text{w.r.t.\@} \ c \right\}.
 \end{align*} 
\end{enumerate}
Suppose for a moment that this has been accomplished.
Then for each $1 \leq k \leq {l}^{2}$, ${\left( {S}^{\ast}_{k} \right)}^{\uparrow {A}_{{l}^{2}}'}$ is non-stationary in ${\[{A}_{{l}^{2}}'\]}^{< {\aleph}_{1}}$ and ${\left( {T}^{\ast}_{k} \right)}^{\uparrow {B}_{{l}^{2}}'}$ is non-stationary in ${\[{B}_{{l}^{2}}'\]}^{< {\aleph}_{1}}$.
Therefore if ${S}^{\ast} = {S}_{{l}^{2}}' \setminus \left( {\bigcup}_{1 \leq k \leq {l}^{2}}{{\left( {S}^{\ast}_{k} \right)}^{\uparrow {A}_{{l}^{2}}'}} \right)$, then $\pr{{A}_{{l}^{2}}'}{{S}^{\ast}} \leq \pr{{A}_{{l}^{2}}'}{{S}_{{l}^{2}}'} \leq \pr{{A}_{0}}{{S}_{0}}$, and if ${T}^{\ast} = {T}_{{l}^{2}}' \setminus \left( {\bigcup}_{1 \leq k \leq {l}^{2}}{{\left( {T}^{\ast}_{k} \right)}^{\uparrow {B}_{{l}^{2}}'}}\right)$, then $\pr{{B}_{{l}^{2}}'}{{T}^{\ast}} \leq \pr{{B}_{{l}^{2}}'}{{T}_{{l}^{2}}'} \leq \pr{{A}_{0}}{{S}_{0}}$.
Choose ${M}^{\ast} \in {S}^{\ast}$ and ${K}^{\ast} \in {T}^{\ast}$.
Apply Lemma \ref{lem:oneislarge} to find $1 \leq k \leq {l}^{2}$ such that ${M}^{\ast}$ is ${i}_{k}$-large in $\pr{{B}_{{l}^{2}}'}{{T}^{\ast}}$ w.r.t.\@ $c$ and ${K}^{\ast}$ is ${j}_{k}$-large in $\pr{{A}_{{l}^{2}}'}{{S}^{\ast}}$ w.r.t.\@ $c$.
Note that ${M}^{\ast} \cap {A}_{k}' \in {S}_{k}' \setminus {S}^{\ast}_{k}$ and ${K}^{\ast} \cap {B}_{k}' \in {T}_{k}' \setminus {T}^{\ast}_{k}$.
By Lemma \ref{lem:largenessprojects}, ${M}^{\ast} \cap {A}_{k}'$ is ${i}_{k}$-large in $\pr{{B}_{k}'}{{T}_{k}'}$ w.r.t.\@ $c$ and ${K}^{\ast} \cap {B}_{k}'$ is ${j}_{k}$-large in $\pr{{A}_{k}'}{{S}_{k}'}$ w.r.t.\@ $c$.
However these facts contradict (2) because they imply that ${S}^{\ast}_{k} \neq \left\{ M \in {S}_{k}': M \ \text{is} \ {i}_{k}\text{-large in} \ \pr{{B}_{k}'}{{T}_{k}'} \ \text{w.r.t.\@} \ c \right\}$ and ${T}^{\ast}_{k} \neq \left\{ K \in {T}_{k}': K \ \text{is} \ {j}_{k}\text{-large in} \ \pr{{A}_{k}'}{{S}_{k}'} \ \text{w.r.t.\@} \ c \right\}$.

To construct such sequences, proceed by induction.
To start, let $\pr{{A}_{0}'}{{S}_{0}'} = \pr{{A}_{{l}^{2}}}{{S}_{{l}^{2}}} = \pr{{B}_{0}'}{{T}_{0}'}$.
Now suppose that $0 \leq k < k + 1 \leq {l}^{2}$ and that $\pr{{A}_{k}'}{{S}_{k}'} \leq \pr{{A}_{{l}^{2}}}{{S}_{{l}^{2}}}$ and $\pr{{B}_{k}'}{{T}_{k}'} \leq \pr{{A}_{{l}^{2}}}{{S}_{{l}^{2}}}$ are given.
Then $\pr{{A}_{k}'}{{S}_{k}'} \leq \pr{{A}_{k + 1}}{{S}_{k + 1}}$ and $\pr{{B}_{k}'}{{T}_{k}'} \leq \pr{{A}_{k + 1}}{{S}_{k + 1}}$.
By the choice of $\pr{{A}_{k + 1}}{{S}_{k + 1}}$, $\pr{\pr{{A}_{k}'}{{S}_{k}'}}{\pr{{B}_{k}'}{{T}_{k}'}}$ is not $\pr{{i}_{k + 1}}{{j}_{k + 1}}$-saturated w.r.t.\@ $c$.
Therefore we can find $\pr{{A}_{k + 1}'}{{S}_{k + 1}'} \leq \pr{{A}_{k}'}{{S}_{k}'}$ and $\pr{{B}_{k + 1}'}{{T}_{k + 1}'} \leq \pr{{B}_{k}'}{{T}_{k}'}$ such that \emph{either} 
\begin{align*}
\left\{ M \in {S}_{k + 1}': M \ \text{is} \ {i}_{k + 1}\text{-large in} \ \pr{{B}_{k + 1}'}{{T}_{k + 1}'} \ \text{w.r.t.\@} \ c \right\}
\end{align*}
is non-stationary in ${\[{A}_{k + 1}'\]}^{< {\aleph}_{1}}$ \emph{or}
\begin{align*}
\left\{ K \in {T}_{k + 1}': K \ \text{is} \ {j}_{k + 1}\text{-large in} \ \pr{{A}_{k + 1}'}{{S}_{k + 1}'} \ \text{w.r.t.\@} \ c \right\}
\end{align*}
is non-stationary in ${\[{B}_{k + 1}'\]}^{< {\aleph}_{1}}$.
If the first alternative happens, then define
\begin{align*}
{S}^{\ast}_{k + 1} = \left\{ M \in {S}_{k + 1}': M \ \text{is} \ {i}_{k + 1}\text{-large in} \ \pr{{B}_{k + 1}'}{{T}_{k + 1}'} \ \text{w.r.t.\@} \ c \right\} \subseteq {S}_{k + 1}'
\end{align*}
and ${T}^{\ast}_{k + 1} = \emptyset$, while if the second alternative occurs, then define ${S}^{\ast}_{k + 1} = \emptyset$ and
\begin{align*}
{T}^{\ast}_{k + 1} = \left\{ K \in {T}_{k + 1}': K \ \text{is} \ {j}_{k + 1}\text{-large in} \ \pr{{A}_{k + 1}'}{{S}_{k + 1}'} \ \text{w.r.t.\@} \ c \right\} \subseteq {T}_{k + 1}'.
\end{align*}
It is clear that $\pr{{A}_{k + 1}'}{{S}_{k + 1}'}$, $\pr{{B}_{k + 1}'}{{T}_{k + 1}'}$, ${S}^{\ast}_{k + 1}$, and ${T}^{\ast}_{k + 1}$ are as required.
This completes the construction and the proof.
\end{proof}
We would like to point out that in certain special circumstances, it is possible to ensure that $i = j$ in Lemma \ref{lem:existssaturated}.
Suppose for a moment that ${X}^{2}$ is a Baire space, that $c$ is Baire measurable, and that the Kuratowski-Ulam theorem is applicable in every open subset of ${X}^{2}$.
Under these circumstances, ${\Q}_{< \delta}$ may be replaced everywhere by the co-ideal of non-meager subsets of $X$.
By Baire measurability, there must be a color $i$ and open sets ${U}_{0}, {U}_{1} \subseteq X$ such that the $i$-colored points are comeager relative to ${U}_{0} \times {U}_{1}$.
By Kuratowski-Ulam, almost all the points in almost all vertical sections of ${U}_{0} \times {U}_{1}$ must have color $i$.
In fact, under these assumptions, the rest of our proof can be completed using the co-ideal of non-meager sets to produce a homeomorphic copy of $\Q$ that is monochromatic in the color $i$.  
This should be compared to a theorem of Todorcevic~\cite{introramsey} saying that if $c: {\[\Q\]}^{2} \rightarrow \nn$ is any continuous coloring, where $\nn$ is given the discrete topology, then there exists a monochromatic $Y \subseteq \Q$ which is homeomorphic to $\Q$.

The next lemma will only be used in the final construction.
It is a simple consequence of the fact that the non-stationary sets form a $\sigma$-ideal.
\begin{Lemma} \label{lem:countablesaturation}
 Suppose $\F \subseteq {\Q}_{< \delta}$ is a countable family so that
 \begin{align*}
  \forall \pr{B}{T} \in \F\[\pr{B}{T} \leq \pr{{A}_{0}}{{S}_{0}}\].
 \end{align*}
 Suppose $k \in l$.
 Let $\pr{A}{S} \leq \pr{{A}_{0}}{{S}_{0}}$ have the property that for any $\pr{A'}{S'} \leq \pr{A}{S}$ and for any $\pr{B}{T} \in \F$, $\left\{ M' \in S': M' \ \text{is} \ k\text{-large in} \ \pr{B}{T} \ \text{w.r.t.\@} \ c \right\}$ is stationary in ${\[A'\]}^{< {\aleph}_{1}}$.
 Then for any $\pr{A'}{S'} \leq \pr{A}{S}$,
 \begin{align*}
   \left\{ M' \in S': \exists \pr{B}{T} \in \F\[M' \ \text{is not} \ k\text{-large in} \ \pr{B}{T} \ \text{w.r.t.\@} \ c\] \right\}
 \end{align*}
 is non-stationary in ${\[A'\]}^{< {\aleph}_{1}}$.
\end{Lemma}
\begin{proof}
  We argue by contradiction.
  If there exists an $\pr{A'}{S'} \leq \pr{A}{S}$ for which the statement of the lemma fails, then there exists a set $S'' \subseteq S'$ which is stationary in ${\[A'\]}^{< {\aleph}_{1}}$ and has the property that for any $M' \in S''$, there exists $\pr{{B}_{M'}}{{T}_{M'}} \in \F$ such that $M'$ is not $k$-large in $\pr{{B}_{M'}}{{T}_{M'}}$ w.r.t.\@ $c$.
  Since $\F$ is a countable set, it follows that there exists $\pr{B}{T} \in \F$ so that ${S}^{\ast} = \{M' \in S'': \pr{B}{T} = \pr{{B}_{M'}}{{T}_{M'}} \}$ is stationary in ${\[A'\]}^{< {\aleph}_{1}}$.
  Thus $\pr{A'}{{S}^{\ast}} \leq \pr{A}{S}$ and so by the hypothesis on $\pr{A}{S}$, $\left\{ M' \in {S}^{\ast}: M' \ \text{is} \ k\text{-large in} \ \pr{B}{T} \ \text{w.r.t.\@} \ c \right\}$ is stationary in ${\[A'\]}^{< {\aleph}_{1}}$.
  In particular this set is non-empty, which contradicts the choice of $S''$, concluding the proof.
\end{proof}
In view of Lemma \ref{lem:existssaturated}, we fix for the remainder of this section pairs $\pr{i}{j} \in l \times l$ and $\pr{{A}_{1}}{{S}_{1}} \leq \pr{{A}_{0}}{{S}_{0}}$ such that for any $\pr{{A}_{2}}{{S}_{2}} \leq \pr{{A}_{1}}{{S}_{1}}$, there exist $\pr{A}{S} \leq \pr{{A}_{2}}{{S}_{2}}$ and $\pr{B}{T} \leq \pr{{A}_{2}}{{S}_{2}}$ such that $\pr{\pr{A}{S}}{\pr{B}{T}}$ is $\pr{i}{j}$-saturated w.r.t.\@ $c$.
We will ensure that all the pairs in the homeomorphic copy of $\Q$ which we are going to construct inside $X$ are colored either $i$ or $j$.

Lemma \ref{lem:mainpdl} is another application of the pressing down lemma.
Lemma \ref{lem:mainsplitting} is proved using Lemmas \ref{lem:mainpdl} and \ref{lem:existssaturated}.
Item (2) of Lemma \ref{lem:mainsplitting} is implied by item (1), but it is stated for emphasis.
\begin{Lemma} \label{lem:mainpdl}
  Suppose $\pr{A}{S} \leq \pr{{A}_{0}}{{S}_{0}}$.
  For each $n \in \omega$, there exists $U$ so that $\{M' \in S: U = {U}_{F(M'), n}\}$ is stationary in ${\[A\]}^{< {\aleph}_{1}}$.
\end{Lemma}
\begin{proof}
   By Lemma \ref{lem:gooddense}, there exists $\pr{B}{T} \leq \pr{A}{S}$ with the property that there exists an uncountable regular cardinal $\chi$ such that $B = H(\chi)$ and for each $M \in T$, $\lc M \rc = {\aleph}_{0}$, $M \prec H(\chi)$, and $\pr{X}{\TT}, \BB \in M$.
   Consider any $M \in T$.
   Then $F(M) \in \overline{M \cap X} \setminus M$.
   So by Lemma \ref{lem:elementary1}, ${\BB}_{\left\{F(M)\right\}} \subseteq M$.
   In particular, $ {U}_{F(M), n} \in M$.
   Thus by the pressing down lemma there exists $U$ such that
   \begin{align*}
    T' = \{M \in T: U = {U}_{F(M), n} \} \subseteq {\[B\]}^{< {\aleph}_{1}}
   \end{align*}
   is stationary.
   So ${T}_{\downarrow A}' \subseteq {\[A\]}^{< {\aleph}_{1}}$ is stationary.
   Since
   \begin{align*}
    {T}_{\downarrow A}' \subseteq \{M' \in S: U = {U}_{F(M'), n}\} \subseteq S \subseteq {\[A\]}^{< {\aleph}_{1}},
   \end{align*}
   $\{M' \in S: U = {U}_{F(M'), n}\}$ is also stationary in ${\[A\]}^{< {\aleph}_{1}}$. 
\end{proof}
\begin{Lemma} \label{lem:mainsplitting}
  Suppose $\pr{A}{S} \leq \pr{{A}_{1}}{{S}_{1}}$ and $n \in \omega$.
  There exist $\pr{A'}{S'}, \pr{B'}{T'} \leq \pr{A}{\widetilde{S}}$ and there exists $U$ satisfying the following:
  \begin{enumerate}
    \item
     for all $M' \in S'$, $U = {U}_{F(M'), n}$ and for all $K' \in T'$, $U = {U}_{F(K'), n}$;
    \item
    for each $M' \in S'$, $F(M') \in U$ and for each $K' \in T'$, $F(K') \in U$;
    \item
    $\pr{\pr{A'}{S'}}{\pr{B'}{T'}}$ is $\pr{i}{j}$-saturated w.r.t.\@ $c$.
  \end{enumerate}
\end{Lemma}
\begin{proof}
 Since $\pr{A}{\widetilde{S}} \leq \pr{A}{S} \leq \pr{{A}_{1}}{{S}_{1}} \leq \pr{{A}_{0}}{{S}_{0}}$, Lemma \ref{lem:mainpdl} applies and implies that there exists $U$ so that ${S}^{\ast} = \{ M \in \widetilde{S}: U = {U}_{F(M), n} \}$ is stationary in ${\[A\]}^{< {\aleph}_{1}}$.
 So $\pr{A}{{S}^{\ast}} \leq \pr{A}{\widetilde{S}}$ and by Lemma \ref{lem:badsmall}, $\pr{A}{\widetilde{\left( {S}^{\ast} \right)}} \leq \pr{A}{{S}^{\ast}}$.
 Choose any ${M}^{\ast} \in \widetilde{\left( {S}^{\ast} \right)}$.
 Then ${S}^{\ast\ast} = \{ M \in {S}^{\ast}: F(M) \in {U}_{F({M}^{\ast}), n} = U \}$ is stationary in ${\[A\]}^{< {\aleph}_{1}}$.
 Therefore, $\pr{A}{{S}^{\ast\ast}} \leq \pr{A}{{S}^{\ast}} \leq \pr{{A}_{1}}{{S}_{1}}$, and by the choice of $\pr{{A}_{1}}{{S}_{1}}$, there exist $\pr{A'}{S'}, \pr{B'}{T'} \leq \pr{A}{{S}^{\ast\ast}}$ such that $\pr{\pr{A'}{S'}}{\pr{B'}{T'}}$ is $\pr{i}{j}$-saturated w.r.t.\@ $c$.
 It is clear that $\pr{A'}{S'}$ and $\pr{B'}{T'}$ are as required.
\end{proof}
\begin{Def} \label{def:maingame}
  Suppose $x \in {A}_{0}$ and $\pr{A}{S} \leq \pr{{A}_{1}}{{S}_{1}}$.
  We will say that \emph{$x$ is an $\pr{i}{j}$-winner in $\pr{A}{S}$} if there exists $M \in \widetilde{S}$ with $F(M) = x$ and there exists a sequence $\langle \pr{\pr{{A}_{x, n}}{{S}_{x, n}}}{\pr{{B}_{x, n}}{{T}_{x, n}}} :n \in \omega \rangle$ satisfying the following conditions:
  \begin{enumerate} 
    \item
    for each $n \in \omega$, $\pr{{A}_{x, n}}{{S}_{x, n}}, \pr{{B}_{x, n}}{{T}_{x, n}} \leq \pr{A}{\widetilde{S}}$, and
    \begin{align*}
     \pr{{A}_{x, n + 1}}{{S}_{x, n + 1}}, \pr{{B}_{x, n + 1}}{{T}_{x, n + 1}} \leq \pr{{A}_{x, n}}{{\widetilde{S}}_{x, n}};
    \end{align*}
    \item
    for each $n \in \omega$, there exists $M \in {\widetilde{S}}_{x, n}$ with $F(M) = x$, and moreover for each $M' \in {S}_{x, n}$, $F(M') \in {U}_{x, n}$ and for each $K' \in {T}_{x, n}$, $F(K') \in {U}_{x, n}$;
    \item
    for each $n \in \omega$, $\pr{\pr{{A}_{x, n}}{{S}_{x, n}}}{\pr{{B}_{x, n}}{{T}_{x, n}}}$ is $\pr{i}{j}$-saturated w.r.t.\@ $c$;
    \item
    for each $n \in \omega$, for each $K' \in {T}_{x, n}$, $F(K') \neq x$ and $c(x, F(K')) = i$.
  \end{enumerate}
\end{Def}
We would like to point out certain features of Definition \ref{def:maingame}.
Intuitively speaking, the sequence of sets $\langle {T}_{x, n}: n \in \omega \rangle$ is converging to the $\pr{i}{j}$-winner $x$.
Moreover $x$ has color $i$ with all of the points in ${T}_{x, n}$ for all $n$, and the pair $\pr{\pr{{B}_{x, k}}{{T}_{x, k}}}{\pr{{B}_{x, n}}{{T}_{x, n}}}$ is $\pr{i}{j}$-saturated for all $n < k$.
These properties of an $\pr{i}{j}$-winner are formulated and proved in Lemma \ref{lem:Bsequence}.
And they are essentially the only properties of an $\pr{i}{j}$-winner that will be used in the final construction.
Thus the condition $\pr{{A}_{x, n}}{{S}_{x, n}}$ is not directly used at all, though it is the reservoir from which the future $\pr{{B}_{x, k}}{{T}_{x, k}}$ are drawn.
Also in item (2) of Definition \ref{def:maingame}, the property that $F(M) = x$ for some $M \in {\widetilde{S}}_{x, n}$ will not be used, though it is automatically ensured by the proof that $\pr{i}{j}$-winners exist.

The next lemma is the key to the final construction.
It asserts that almost every point in any condition below $\pr{{A}_{1}}{{S}_{1}}$ is an $\pr{i}{j}$-winner in that condition.
Its proof appeals to Theorem \ref{thm:banachmazur}, and it is the only place in the proof of Theorem \ref{thm:main} where the assumption that $\delta$ is Woodin or strongly compact is essential.  
\begin{Lemma} \label{lem:mainlemma}
  For any $\pr{A}{S} \leq \pr{{A}_{1}}{{S}_{1}}$,
  \begin{align*}
   \left\{ M \in S: F(M) \ \text{is not an} \ \pr{i}{j}\text{-winner in} \ \pr{A}{S} \right\}
  \end{align*}
  is non-stationary in ${\[A\]}^{< {\aleph}_{1}}$.
\end{Lemma}
\begin{proof}
Suppose not.
Then
\begin{align*}
 S' = \left\{ M \in \widetilde{S}: F(M) \ \text{is not an} \ \pr{i}{j}\text{-winner in} \ \pr{A}{S} \right\}
\end{align*}
is stationary in ${\[A\]}^{< {\aleph}_{1}}$.
Thus $\pr{A}{S'} \leq \pr{A}{\widetilde{S}} \leq \pr{A}{S} \leq \pr{{A}_{1}}{{S}_{1}}$.
Applying Lemma \ref{lem:mainsplitting} with $\pr{A}{S'}$ in place of $\pr{A}{S}$, choose $\pr{{C}_{0}}{{R}_{0}'}, \pr{{B}_{0}}{{T}_{0}} \leq \pr{A}{S'}$ and ${U}_{0}$ so that for each $M' \in {R}_{0}'$, $F(M') \in {U}_{0} = {U}_{F(M'), 0}$, for each $K' \in {T}_{0}$, $F(K') \in {U}_{0} = {U}_{F(K'), 0}$, and $\pr{\pr{{C}_{0}}{{R}_{0}'}}{\pr{{B}_{0}}{{T}_{0}}}$ is $\pr{i}{j}$-saturated w.r.t.\@ $c$.
In particular, ${R}_{0} = \left\{ M' \in {R}_{0}': M' \ \text{is} \ i\text{-large in} \ \pr{{B}_{0}}{{T}_{0}} \ \text{w.r.t.\@} \ c \right\}$ is stationary in ${\[{C}_{0}\]}^{< {\aleph}_{1}}$.
Now define a strategy for Empty in $\Game(\delta)$ as follows.
Suppose that $\sigma$ is a partial run of $\Game(\delta)$ with $\lc \sigma \rc = 2n$ (for some $n \in \omega$), during which Empty has followed his strategy.
If $n = 0$, then Empty plays $\pr{{C}_{0}}{{\widetilde{R}}_{0}}$.
If $n > 0$, then $\sigma(2n-1) \leq \sigma(0) = \pr{{C}_{0}}{{\widetilde{R}}_{0}} \leq \pr{{A}_{1}}{{S}_{1}}$.
Applying Lemma \ref{lem:mainsplitting} with $\sigma(2n-1)$ in place of $\pr{A}{S}$, Empty chooses $\pr{{C}_{\sigma}}{{R}_{\sigma}'}, \pr{{B}_{\sigma}}{{T}_{\sigma}} \leq \sigma(2n-1)$ and ${U}_{n}$ so that for each $M' \in {R}_{\sigma}'$, $F(M') \in {U}_{n} = {U}_{F(M'), n}$, for each $K' \in {T}_{\sigma}$, $F(K') \in {U}_{n} = {U}_{F(K'), n}$, and $\pr{\pr{{C}_{\sigma}}{{R}_{\sigma}'}}{\pr{{B}_{\sigma}}{{T}_{\sigma}}}$ is $\pr{i}{j}$-saturated w.r.t.\@ $c$.
In particular, ${R}_{\sigma} = \left\{ M' \in {R}_{\sigma}': M' \ \text{is} \ i\text{-large in} \ \pr{{B}_{\sigma}}{{T}_{\sigma}} \ \text{w.r.t.\@} \ c \right\}$ is stationary in ${\[{C}_{\sigma}\]}^{< {\aleph}_{1}}$.
Empty then plays $\pr{{C}_{\sigma}}{{\widetilde{R}}_{\sigma}} \leq \sigma(2n-1)$ as the $2n$-th move of this run.
This concludes the definition of a strategy for Empty in $\Game(\delta)$.

Since Empty does not have a winning strategy, there is a complete run of $\Game(\delta)$ in which Empty follows the strategy defined above and looses.
Therefore there exist sequences $\langle \pr{{C}_{n}}{{R}_{n}}: n \in \omega \rangle$, $\langle \pr{\pr{{C}_{2n}}{{R}_{2n}'}}{\pr{{B}_{2n}}{{T}_{2n}}}: n \in \omega \rangle$, and $\seq{U}{n}{\in}{\omega}$ satisfying the following:
\begin{enumerate}
  \item
  Non-Empty wins the run of $\Game(\delta)$ given by
  \begin{align*}
    \begin{tabular}{c| c c c c c c c}
      $\textrm{Empty}$& $\pr{{C}_{0}}{{\widetilde{R}}_{0}}$ &  & $\pr{{C}_{2}}{{\widetilde{R}}_{2}}$& & $\dotsb$ & &\\
      \hline
      $\textrm{Non-Empty}$& & $\pr{{C}_{1}}{{R}_{1}}$& & $\pr{{C}_{3}}{{R}_{3}}$ & & $\dotsb$
    \end{tabular}
  \end{align*}
  \item
  for each $n \in \omega$, for each $M' \in {R}_{2n}'$, $F(M') \in {U}_{n} = {U}_{F(M'), n}$, for each $K' \in {T}_{2n}$, $F(K') \in {U}_{n} = {U}_{F(K'), n}$, and $\pr{\pr{{C}_{2n}}{{R}_{2n}'}}{\pr{{B}_{2n}}{{T}_{2n}}}$ is $\pr{i}{j}$-saturated w.r.t.\@ $c$;
  \item
  for each $n \in \omega$, ${R}_{2n} = \left\{ M' \in {R}_{2n}': M' \ \text{is} \ i\text{-large in} \ \pr{{B}_{2n}}{{T}_{2n}} \ \text{w.r.t.\@} \ c \right\}$ is stationary in ${\[{C}_{2n}\]}^{< {\aleph}_{1}}$;
  \item
  for each $n > 0$, $\pr{{C}_{2n}}{{R}_{2n}'}, \pr{{B}_{2n}}{{T}_{2n}} \leq \pr{{C}_{2n-1}}{{R}_{2n-1}}$. 
\end{enumerate}
There is a sequence $\seq{M}{n}{\in}{\omega}$ so that
\begin{align*}
 &\forall n \in \omega\[{M}_{2n} \in {\widetilde{R}}_{2n} \ \text{and} \ {M}_{2n+1} \in {R}_{2n+1}\] \ \text{and}\\
 &\forall n \leq k < \omega\[{M}_{n} = {M}_{k} \cap {C}_{n}\]
\end{align*}
because Non-Empty wins.
Define $x = F({M}_{0})$.
Note that for any $n > 0$, $F({M}_{2n}) = F({M}_{0}) = x$.
Furthermore, ${M}_{0} \cap A \in S'$, which means that ${M}_{0} \cap A \in \widetilde{S}$ and $x = F({M}_{0} \cap A)$ is not an $\pr{i}{j}$-winner in $\pr{A}{S}$.
We will get a contradiction by showing that $x$ is an $\pr{i}{j}$-winner in $\pr{A}{S}$.

First note that if we let $M = {M}_{0} \cap A$, then $M \in \widetilde{S}$ and $x = F(M)$.
Now define a sequence $\langle \pr{\pr{{A}_{x, n}}{{S}_{x, n}}}{\pr{{B}_{x, n}}{{T}_{x, n}}}: n \in \omega \rangle$ as follows.
Fix $n \in \omega$ and define $\pr{{A}_{x, n}}{{S}_{x, n}} = \pr{{C}_{2n}}{{R}_{2n}}$.
Note that ${M}_{2n} \in {R}_{2n}$, whence ${M}_{2n} \in {R}_{2n}'$ and ${M}_{2n}$ is $i$-large in $\pr{{B}_{2n}}{{T}_{2n}}$ w.r.t.\@ $c$, which means that
\begin{align*}
 {T}_{x, n} = \{K' \in {T}_{2n}: F({M}_{2n}) \neq F(K') \ \text{and} \ c(F({M}_{2n}), F(K')) = i\}
\end{align*}
is stationary in ${\[{B}_{2n}\]}^{< {\aleph}_{1}}$.
Defining ${B}_{x, n} = {B}_{2n}$, we have that $\pr{{B}_{x, n}}{{T}_{x, n}} = \pr{{B}_{2n}}{{T}_{x, n}} \leq \pr{{B}_{2n}}{{T}_{2n}}$.
Moreover by the definition of ${T}_{x, n}$, for any $K' \in {T}_{x, n}$, $x \neq F(K')$ and $c(x, F(K')) = i$, which is what (4) of Definition \ref{def:maingame} says.
Also $\pr{\pr{{C}_{2n}}{{R}_{2n}'}}{\pr{{B}_{2n}}{{T}_{2n}}}$ is $\pr{i}{j}$-saturated w.r.t.\@ $c$, $\pr{{A}_{x, n}}{{S}_{x, n}} = \pr{{C}_{2n}}{{R}_{2n}} \leq \pr{{C}_{2n}}{{R}_{2n}'}$, and $\pr{{B}_{x, n}}{{T}_{x, n}} \leq \pr{{B}_{2n}}{{T}_{2n}}$, which implies that
\begin{align*}
 \pr{\pr{{A}_{x, n}}{{S}_{x, n}}}{\pr{{B}_{x, n}}{{T}_{x, n}}}
\end{align*}
is $\pr{i}{j}$-saturated w.r.t.\@ $c$, satisfying (3) of Definition \ref{def:maingame}.
Next, note that ${M}_{2n} \in {\widetilde{R}}_{2n} = {\widetilde{S}}_{x, n}$ and $F({M}_{2n}) = x$.
Note also that since ${M}_{2n} \in {R}_{2n}'$, ${U}_{n} = {U}_{F({M}_{2n}), n} = {U}_{x, n}$.
Moreover for any $M' \in {S}_{x, n}$, $F(M') \in {U}_{n} = {U}_{x, n}$, and for any $K' \in {T}_{x, n}$, $F(K') \in {U}_{n} = {U}_{x, n}$.
Hence (2) of Definition \ref{def:maingame} is satisfied.
Furthermore, if $n = 0$, then $\pr{{A}_{x, n}}{{S}_{x, n}} \leq \pr{{C}_{2n}}{{R}_{2n}'} = \pr{{C}_{0}}{{R}_{0}'} \leq \pr{A}{\widetilde{S}}$ and $\pr{{B}_{x, n}}{{T}_{x, n}} \leq \pr{{B}_{2n}}{{T}_{2n}} = \pr{{B}_{0}}{{T}_{0}} \leq \pr{A}{S'} \leq \pr{A}{\widetilde{S}}$.
If $n > 0$, then $\pr{{C}_{2n}}{{R}_{2n}'}, \pr{{B}_{2n}}{{T}_{2n}} \leq \pr{{C}_{2n-1}}{{R}_{2n-1}} \leq \pr{{C}_{0}}{{\widetilde{R}}_{0}} \leq \pr{A}{\widetilde{S}}$, and so $\pr{{A}_{x, n}}{{S}_{x, n}} \leq \pr{{C}_{2n}}{{R}_{2n}'} \leq \pr{A}{\widetilde{S}}$ and $\pr{{B}_{x, n}}{{T}_{x, n}} \leq \pr{{B}_{2n}}{{T}_{2n}} \leq \pr{A}{\widetilde{S}}$.
Thus $\pr{{A}_{x, n}}{{S}_{x, n}}, \pr{{B}_{x, n}}{{T}_{x, n}} \leq \pr{A}{\widetilde{S}}$ always holds.
Finally we have that $\pr{{A}_{x, n+1}}{{S}_{x, n+1}} \leq \pr{{C}_{2n+2}}{{R}_{2n+2}'} \leq \pr{{C}_{2n+1}}{{R}_{2n+1}} \leq \pr{{A}_{x, n}}{{\widetilde{S}}_{x, n}}$ and that
\begin{align*}
 \pr{{B}_{x, n+1}}{{T}_{x, n+1}} \leq \pr{{B}_{2n+2}}{{T}_{2n+2}} \leq \pr{{C}_{2n+1}}{{R}_{2n+1}} \leq \pr{{A}_{x, n}}{{\widetilde{S}}_{x, n}}.
\end{align*}
Thus $\pr{{A}_{x, n + 1}}{{S}_{x, n + 1}}, \pr{{B}_{x, n + 1}}{{T}_{x, n + 1}} \leq \pr{{A}_{x, n}}{{\widetilde{S}}_{x, n}}$ holds, and so (1) of Definition \ref{def:maingame} holds.

This concludes the verification that $x$ is an $\pr{i}{j}$-winner in $\pr{A}{S}$.
Since this yields a contradiction, the proof is complete.
\end{proof}
\begin{Lemma} \label{lem:Bsequence}
 Suppose $x \in {A}_{0}$ and that $\pr{A}{S} \leq \pr{{A}_{1}}{{S}_{1}}$.
 If $x$ is an $\pr{i}{j}$-winner in $\pr{A}{S}$, then there exists a sequence $\langle \pr{{B}_{x, n}}{{T}_{x, n}}: n \in \omega \rangle$ such that the following hold for each $n \in \omega$:
 \begin{enumerate}
   \item
   $\pr{{B}_{x, n}}{{T}_{x, n}} \leq \pr{A}{S}$;
   \item
   for each $K' \in {T}_{x, n}$, $F(K') \in {U}_{x, n}$;
   \item
   for any $n < k < \omega$, $\pr{\pr{{B}_{x, k}}{{T}_{x, k}}}{\pr{{B}_{x, n}}{{T}_{x, n}}}$ is $\pr{i}{j}$-saturated w.r.t.\@ $c$; 
   \item
   for each $K' \in {T}_{x, n}$, $F(K') \neq x$ and $c(x, F(K')) = i$.
 \end{enumerate}
\end{Lemma}
\begin{proof}
  By the definition of an $\pr{i}{j}$-winner in $\pr{A}{S}$, there exists a sequence
  \begin{align*}
   \langle \pr{\pr{{A}_{x, n}}{{S}_{x, n}}}{\pr{{B}_{x, n}}{{T}_{x, n}}} :n \in \omega \rangle
  \end{align*}
  satisfying (1)--(4) of Definition \ref{def:maingame}.
  We argue that $\langle \pr{{B}_{x, n}}{{T}_{x, n}}: n \in \omega \rangle$ has the required properties.
  Indeed, from (1) of Definition \ref{def:maingame}, $\pr{{B}_{x, n}}{{T}_{x, n}} \leq \pr{A}{\widetilde{S}} \leq \pr{A}{S}$, for each $n \in \omega$.
  Next, (2) and (4) of this lemma follow from (2) and (4) of Definition \ref{def:maingame} respectively.
  Finally, for any $n \in \omega$ and for any $n < k < \omega$, $\pr{{B}_{x, k}}{{T}_{x, k}} \leq \pr{{A}_{x, n}}{{S}_{x, n}}$ by (1) of Definition \ref{def:maingame}.
  If $n \in \omega$, then $\pr{\pr{{A}_{x, n}}{{S}_{x, n}}}{\pr{{B}_{x, n}}{{T}_{x, n}}}$ is $\pr{i}{j}$-saturated w.r.t.\@ $c$, whence for any $n < k < \omega$, $\pr{\pr{{B}_{x, k}}{{T}_{x, k}}}{\pr{{B}_{x, n}}{{T}_{x, n}}}$ is $\pr{i}{j}$-saturated w.r.t.\@ $c$.
\end{proof}
\begin{Def} \label{def:leaf}
  If $P \subseteq {\omega}^{< \omega}$ is a downwards closed subtree, we say that $\sigma$ is a \emph{leaf node of $P$} if $\sigma \in P$, but there is no $m \in \omega$ for which ${\sigma}^{\frown}{\langle m \rangle} \in P$.
  $L(P)$ will denote the collection of all leaf nodes of $P$.
  $N(P)$ will denote $P \setminus L(P)$.
  Thus $P = L(P) \cup N(P)$.
  
  If $\sigma, \tau \in {\omega}^{< \omega}$ are incomparable, then
  \begin{align*}
   \Delta(\sigma, \tau) = \min\left\{ m \in \dom(\sigma) \cap \dom(\tau): \sigma(m) \neq \tau(m) \right\}.
  \end{align*}
  We say $\lex{\sigma}{\tau}$ if $\sigma$ and $\tau$ are incomparable and $\sigma(\Delta(\sigma, \tau)) < \tau(\Delta(\sigma, \tau))$.
\end{Def}
We are now ready to prove the main theorem.
We will organize the construction of the homeomorphic copy of $\Q$ by associating every node of the tree ${\omega}^{< \omega}$ to a point in the copy.
This makes certain features of the construction easier to visualize.
For instance, the points associated to the successors of a node converge to the point associated to that node.
Since the construction is inductive, the homeomorphic copy of $\Q$ is naturally well-ordered by the order in which the points are chosen.
Our scheme explicitly displays the interplay between this well-ordering and the lexicographic ordering of the tree, as well as the correspondence between this interplay and the colors $i$ and $j$.
Of course we know from Sierpi{\' n}ski's example that such a close correspondence is unavoidable.
The sequence of trees $\seq{P}{m}{\in}{\omega}$ in the proof of Theorem \ref{thm:main} below serves as a bookkeeping device ensuring that once a point has been chosen, all of its neighborhoods are eventually considered and met.   
\begin{Theorem} \label{thm:main}
  There is a non-empty countable $Y \subseteq X$ such that $Y$ is dense in itself and $c''{\[Y\]}^{2} \subseteq \{i, j\}$.
\end{Theorem}
\begin{proof}
We may choose a sequence $\seq{P}{m}{\in}{\omega}$ satisfying the following conditions:
\begin{enumerate}[series=main]
  \item
  for each $m \in \omega$, ${P}_{m} \subseteq {\omega}^{< \omega}$ is a non-empty downwards closed subtree of finite height;
  \item
  for each $m \in \omega$ there exists ${\sigma}_{m} \in L({P}_{m})$ such that
  \begin{align*}
   {P}_{m + 1} = {P}_{m} \cup \left\{ {\left( {\sigma}_{m} \right)}^{\frown}{\langle n \rangle}: n \in \omega \right\};
  \end{align*}
  \item
  ${P}_{0} = \{\emptyset\}$ and ${\omega}^{< \omega} = {\bigcup}_{n \in \omega}{{P}_{n}}$.
\end{enumerate}
It is clear that for each $m \in \omega$, $L({P}_{m + 1}) = \left( L({P}_{m}) \setminus \{{\sigma}_{m}\} \right) \cup \left\{ {\left( {\sigma}_{m} \right)}^{\frown}{\langle n \rangle}: n \in \omega \right\}$ and that $N({P}_{m + 1}) = N({P}_{m}) \cup \{{\sigma}_{m}\}$.
Also if $m < m' < \omega$, then ${\sigma}_{m} \neq {\sigma}_{m'}$ and ${\sigma}_{m} \in N({P}_{m'})$.
Finally, observe that for each $\sigma \in {\omega}^{< \omega}$, there exists $m \in \omega$ with $\sigma = {\sigma}_{m}$, and that $m + 1$ is the minimal ${m}^{\ast} \in \omega$ with $\sigma \in N({P}_{{m}^{\ast}})$.
We will construct two sequences $\langle {x}_{m+1}: m \in \omega \rangle$ and $\seq{F}{m}{\in}{\omega}$ such that the following conditions hold at each $m \in \omega$:
\begin{enumerate}[resume=main]
  \item
  ${x}_{m+1} \in X$ and ${F}_{m}: L({P}_{m}) \rightarrow {\Q}_{< \delta}$; for a $\sigma \in L({P}_{m})$, we will write $\pr{{B}_{m, \sigma}}{{T}_{m, \sigma}}$ instead of ${F}_{m}(\sigma)$;
  \item
  for each $\sigma \in L({P}_{m})$, $\pr{{B}_{m, \sigma}}{{T}_{m, \sigma}} \leq \pr{{A}_{1}}{{S}_{1}}$, and furthermore for each $m' \leq m$ and for each $\sigma \in L({P}_{m'}) \cap L({P}_{m})$, ${T}_{m, \sigma} \subseteq {T}_{m', \sigma}$; 
  \item
  for each $m' < m$ and for each $\sigma \in L({P}_{m})$, if ${\sigma}_{m'} \subsetneq \sigma$, then for each $K \in {T}_{m, \sigma}$, $F(K) \neq {x}_{m'+1}$ and $c(F(K), {x}_{m'+1}) = i$;
  if $\lex{{\sigma}_{m'}}{\sigma}$, then for each $K \in {T}_{m, \sigma}$, $F(K) \neq {x}_{m'+1}$ and $c(F(K), {x}_{m'+1}) = j$;
  if $\lex{\sigma}{{\sigma}_{m'}}$, then for each $K \in {T}_{m, \sigma}$, $F(K) \neq {x}_{m'+1}$ and $c(F(K), {x}_{m'+1}) = i$;
  \item
  for any $\sigma, \tau \in L({P}_{m})$, if $\lex{\sigma}{\tau}$, then $\pr{\pr{{B}_{m, \tau}}{{T}_{m, \tau}}}{\pr{{B}_{m, \sigma}}{{T}_{m, \sigma}}}$ is $\pr{i}{j}$-saturated w.r.t.\@ $c$;
  \item
  there exists $K \in {T}_{m, {\sigma}_{m}}$ so that ${x}_{m+1} = F(K)$;
  \item
  for each $n \in \omega$, $\left\langle {B}_{(m+1), \left( {\left( {\sigma}_{m} \right)}^{\frown}{\langle n \rangle} \right)}, {T}_{(m+1), \left( {\left( {\sigma}_{m} \right)}^{\frown}{\langle n \rangle} \right)} \right\rangle \leq \pr{{B}_{m, {\sigma}_{m}}}{{T}_{m, {\sigma}_{m}}}$, and furthermore for each $K \in {T}_{(m+1), \left( {\left( {\sigma}_{m} \right)}^{\frown}{\langle n \rangle} \right)}$, $F(K) \in {U}_{{x}_{m+1}, n}$. 
\end{enumerate}
Suppose for a moment that these two sequences can be built.
Define $Y = \{{x}_{m+1}: m \in \omega\}$.
Clearly $Y \subseteq X$, $Y$ is countable, and $Y$ is non-empty.
We first verify that $Y$ is dense in itself.
Indeed, fix $m, n \in \omega$.
We must find some $m' \in \omega$ for which ${x}_{m'+1} \in {U}_{{x}_{m+1}, n}$ and ${x}_{m'+1} \neq {x}_{m+1}$.
Put $\tau = {\left( {\sigma}_{m} \right)}^{\frown}{\langle n \rangle}$.
Then $\tau \in L({P}_{m+1})$.
Let $m' \in \omega$ be so that $\tau = {\sigma}_{m'}$.
It is easy to see that $m+1 \leq m'$.
By (9), for each $K \in {T}_{m+1, \tau}$, $F(K) \in {U}_{{x}_{m+1}, n}$.
By (6) applied to $m < m'$ and $\tau \in L({P}_{m'})$, since ${\sigma}_{m} \subsetneq \tau$, we have that for each $K \in {T}_{m', \tau}$, $F(K) \neq {x}_{m+1}$.
By (5) applied to $m+1 \leq m'$ and $\tau \in L({P}_{m+1}) \cap L({P}_{m'})$, we have that ${T}_{m', \tau} \subseteq {T}_{m+1, \tau}$.
Finally by (8) applied to $m'$, we have that there exists $K \in {T}_{m', \tau}$ so that ${x}_{m'+1} = F(K)$.
Thus ${x}_{m'+1} = F(K) \neq {x}_{m+1}$.
Also $K \in {T}_{m+1, \tau}$, whence ${x}_{m'+1} = F(K) \in {U}_{{x}_{m+1}, n}$, as needed.
This verifies that $Y$ is dense in itself.
We next check that $c''{\[Y\]}^{2} \subseteq \{i, j\}$.
Consider any $m' < m < \omega$.
We will verify that ${x}_{m'+1} \neq {x}_{m+1}$ and that $c({x}_{m'+1}, {x}_{m+1}) \in \{i, j\}$.
Apply (8) to find $K \in {T}_{m, {\sigma}_{m}}$ so that ${x}_{m+1} = F(K)$.
We see that ${\sigma}_{m'} \neq {\sigma}_{m}$, that ${\sigma}_{m'} \in N({P}_{m})$, and that ${\sigma}_{m} \in L({P}_{m})$.
In particular, we cannot have ${\sigma}_{m} \subseteq {\sigma}_{m'}$.
Hence by (6), we have the following three possibilities: if ${\sigma}_{m'} \subsetneq {\sigma}_{m}$, then ${x}_{m+1} \neq {x}_{m'+1}$ and $c({x}_{m+1}, {x}_{m'+1}) = i$; if $\lex{{\sigma}_{m'}}{{\sigma}_{m}}$, then ${x}_{m+1} \neq {x}_{m'+1}$ and $c({x}_{m+1}, {x}_{m'+1}) = j$; if $\lex{{\sigma}_{m}}{{\sigma}_{m'}}$, then ${x}_{m+1} \neq {x}_{m'+1}$ and $c({x}_{m+1}, {x}_{m'+1}) = i$.
This is as required.

To finish the proof, it suffices to construct sequences $\langle {x}_{m+1}: m \in \omega \rangle$ and $\seq{F}{m}{\in}{\omega}$ satisfying (4)--(9).
We do this by induction.
So fix ${m}^{\ast} \in \omega$ and assume that $\langle {x}_{m'+1}: m' < m'+1 < {m}^{\ast} \rangle$ and $\langle {F}_{m'}: m' < {m}^{\ast} \rangle$ have been defined.
We will define ${F}_{{m}^{\ast}}$ and if ${m}^{\ast} \neq 0$, then also ${x}_{{m}^{\ast}}$.
Since $L({P}_{0}) = {P}_{0} = \{\emptyset\}$, when ${m}^{\ast} = 0$, we only need to ensure that $\pr{{B}_{0, \emptyset}}{{T}_{0, \emptyset}}$ is defined and that it is below $\pr{{A}_{1}}{{S}_{1}}$.
So we define $\pr{{B}_{0, \emptyset}}{{T}_{0, \emptyset}} = \pr{{A}_{1}}{{S}_{1}}$.
Now suppose that ${m}^{\ast} = m+1$, for some $m \in \omega$.
Note that since ${\sigma}_{m} \in L({P}_{m})$, every $\sigma \in L({P}_{m}) \setminus \{{\sigma}_{m}\}$ is incomparable to ${\sigma}_{m}$.
Therefore $L({P}_{m}) \setminus \{{\sigma}_{m}\} = {\GG}_{0} \cup {\GG}_{1}$, where ${\GG}_{0} = \{\sigma \in L({P}_{m}): \lex{{\sigma}_{m}}{\sigma}\}$ and ${\GG}_{1} = \{\sigma \in L({P}_{m}): \lex{\sigma}{{\sigma}_{m}}\}$.
Applying Lemma \ref{lem:countablesaturation}, we conclude that
\begin{align*}
 \left\{ K' \in {T}_{m, {\sigma}_{m}}: \exists \sigma \in {\GG}_{0} \[K' \ \text{is not} \ j\text{-large in} \ \pr{{B}_{m, \sigma}}{{T}_{m, \sigma}} \ \text{w.r.t.\@} \ c \]\right\}
\end{align*}
is non-stationary in ${\[{B}_{m, {\sigma}_{m}}\]}^{< {\aleph}_{1}}$ and also that
\begin{align*}
 \left\{ K' \in {T}_{m, {\sigma}_{m}}: \exists \sigma \in {\GG}_{1}\[K' \ \text{is not} \ i\text{-large in} \ \pr{{B}_{m, \sigma}}{{T}_{m, \sigma}} \ \text{w.r.t.\@} \ c \]\right\}
\end{align*}
is non-stationary in ${\[{B}_{m, {\sigma}_{m}}\]}^{< {\aleph}_{1}}$.
Further, Lemma \ref{lem:mainlemma} tells us that
\begin{align*}
 \left\{ K' \in {T}_{m, {\sigma}_{m}}: F(K') \ \text{is not an} \ \pr{i}{j}\text{-winner in} \ \pr{{B}_{m, {\sigma}_{m}}}{{T}_{m, {\sigma}_{m}}} \right\}
\end{align*}
is non-stationary in ${\[{B}_{m, {\sigma}_{m}}\]}^{< {\aleph}_{1}}$.
Therefore we may choose $K' \in {T}_{m, {\sigma}_{m}}$ such that the following things are satisfied: $\forall \sigma \in {\GG}_{1}\[K' \ \text{is} \ i\text{-large in} \ \pr{{B}_{m, \sigma}}{{T}_{m, \sigma}} \ \text{w.r.t.\@} \ c \]$, $\forall \sigma \in {\GG}_{0} \[K' \ \text{is} \ j\text{-large in} \ \pr{{B}_{m, \sigma}}{{T}_{m, \sigma}} \ \text{w.r.t.\@} \ c \]$, and $F(K')$ is an $\pr{i}{j}$-winner in $\pr{{B}_{m, {\sigma}_{m}}}{{T}_{m, {\sigma}_{m}}}$.
Define ${x}_{m+1} = F(K') = F(K' \cap {A}_{0}) \in X$.
By Lemma \ref{lem:Bsequence}, there exists a sequence $\left\langle \left\langle {B}_{(m+1), \left( {\left( {\sigma}_{m} \right)}^{\frown}{\langle n \rangle} \right)}, {T}_{(m+1), \left( {\left( {\sigma}_{m} \right)}^{\frown}{\langle n \rangle} \right)} \right\rangle: n \in \omega \right\rangle$ such that the following hold for each $n \in \omega$:
 \begin{enumerate}[resume=main]
   \item
   $\left\langle {B}_{(m+1), \left( {\left( {\sigma}_{m} \right)}^{\frown}{\langle n \rangle} \right)}, {T}_{(m+1), \left( {\left( {\sigma}_{m} \right)}^{\frown}{\langle n \rangle} \right)} \right\rangle \leq \pr{{B}_{m, {\sigma}_{m}}}{{T}_{m, {\sigma}_{m}}}$;
   \item
   for each $K \in {T}_{(m+1), \left( {\left( {\sigma}_{m} \right)}^{\frown}{\langle n \rangle} \right)}$, $F(K) \in {U}_{{x}_{m+1}, n}$;
   \item
   for any $n < k < \omega$,
   \begin{align*}
    \pr{\pr{{B}_{(m+1), ({\left( {\sigma}_{m} \right)}^{\frown}{\langle k \rangle})}}{{T}_{(m+1), ({\left( {\sigma}_{m} \right)}^{\frown}{\langle k \rangle})}}}{\pr{{B}_{(m+1), ({\left( {\sigma}_{m} \right)}^{\frown}{\langle n \rangle})}}{{T}_{(m+1), ({\left( {\sigma}_{m} \right)}^{\frown}{\langle n \rangle})}}}
   \end{align*}
   is $\pr{i}{j}$-saturated w.r.t.\@ $c$; 
   \item
   for each $K \in {T}_{(m+1), \left( {\left( {\sigma}_{m} \right)}^{\frown}{\langle n \rangle} \right)}$, $F(K) \neq {x}_{m+1}$ and $c({x}_{m+1}, F(K)) = i$.
 \end{enumerate}
 For each $\sigma \in L({P}_{m}) \setminus \{{\sigma}_{m}\}$, if $\sigma \in {\GG}_{0}$, then define ${B}_{m+1, \sigma} = {B}_{m, \sigma}$ and
 \begin{align*}
  {T}_{m+1, \sigma} = \left\{ K \in {T}_{m, \sigma}: F(K') \neq F(K) \ \text{and} \ c(F(K'), F(K)) = j \right\},
 \end{align*}
 which is a stationary subset of ${\[{B}_{m, \sigma}\]}^{< {\aleph}_{1}}$.
 If $\sigma \in {\GG}_{1}$, then set ${B}_{m+1, \sigma} = {B}_{m, \sigma}$ and
 \begin{align*}
  {T}_{m+1, \sigma} = \left\{ K \in {T}_{m, \sigma}: F(K') \neq F(K) \ \text{and} \ c(F(K'), F(K)) = i \right\},
 \end{align*}
 which is a stationary subset of ${\[{B}_{m, \sigma}\]}^{< {\aleph}_{1}}$.
 Note that for all $\sigma \in L({P}_{m}) \setminus \{{\sigma}_{m}\}$, $\pr{{B}_{m+1, \sigma}}{{T}_{m+1, \sigma}} \leq \pr{{B}_{m, \sigma}}{{T}_{m, \sigma}}$.
 This finishes the definition of ${F}_{m+1}$ and ${x}_{m+1}$.
 It is simple to verify (4), (5), (8), and (9).
 We will go through the verification of (6) and (7).
 To check (7), fix any $\sigma, \tau \in L({P}_{m+1})$ and suppose that $\lex{\sigma}{\tau}$.
 If $\sigma, \tau \in L({P}_{m}) \setminus \{{\sigma}_{m}\}$, then the induction hypothesis applies and implies that $\pr{\pr{{B}_{m, \tau}}{{T}_{m, \tau}}}{\pr{{B}_{m, \sigma}}{{T}_{m, \sigma}}}$ is $\pr{i}{j}$-saturated w.r.t.\@ $c$.
 Since we have $\pr{{B}_{m+1, \tau}}{{T}_{m+1, \tau}} \leq \pr{{B}_{m, \tau}}{{T}_{m, \tau}}$ and $\pr{{B}_{m+1, \sigma}}{{T}_{m+1, \sigma}} \leq \pr{{B}_{m, \sigma}}{{T}_{m, \sigma}}$, it follows that $\pr{\pr{{B}_{m+1, \tau}}{{T}_{m+1, \tau}}}{\pr{{B}_{m+1, \sigma}}{{T}_{m+1, \sigma}}}$ is $\pr{i}{j}$-saturated w.r.t.\@ $c$.
 Next if $\sigma = {({\sigma}_{m})}^{\frown}{\langle n \rangle}$ and $\tau = {({\sigma}_{m})}^{\frown}{\langle k \rangle}$ for some $n, k \in \omega$, then $n < k$, and by (12), $\pr{\pr{{B}_{m+1, \tau}}{{T}_{m+1, \tau}}}{\pr{{B}_{m+1, \sigma}}{{T}_{m+1, \sigma}}}$ is $\pr{i}{j}$-saturated w.r.t.\@ $c$.
 Now suppose that $\sigma \in L({P}_{m}) \setminus \{{\sigma}_{m}\}$ and that $\tau = {({\sigma}_{m})}^{\frown}{\langle n \rangle}$, for some $n \in \omega$.
 Then $\lex{\sigma}{{\sigma}_{m}}$, and since $\sigma, {\sigma}_{m} \in L({P}_{m})$, the induction hypothesis applies and implies that $\pr{\pr{{B}_{m, {\sigma}_{m}}}{{T}_{m, {\sigma}_{m}}}}{\pr{{B}_{m, \sigma}}{{T}_{m, \sigma}}}$ is $\pr{i}{j}$-saturated w.r.t.\@ $c$.
 Since we know that $\pr{{B}_{m+1, \sigma}}{{T}_{m+1, \sigma}} \leq \pr{{B}_{m, \sigma}}{{T}_{m, \sigma}}$ and $\left\langle {B}_{m+1, \tau}, {T}_{m+1, \tau} \right\rangle \leq \pr{{B}_{m, {\sigma}_{m}}}{{T}_{m, {\sigma}_{m}}}$, we conclude that $\pr{\left\langle {B}_{m+1, \tau}, {T}_{m+1, \tau} \right\rangle}{\pr{{B}_{m+1, \sigma}}{{T}_{m+1, \sigma}}}$ is also $\pr{i}{j}$-saturated w.r.t.\@ $c$.
 In the case when $\sigma = {({\sigma}_{m})}^{\frown}{\langle n \rangle}$ for some $n \in \omega$ and $\tau \in L({P}_{m}) \setminus \{{\sigma}_{m}\}$, we have that $\lex{{\sigma}_{m}}{\tau}$.
 Since ${\sigma}_{m}, \tau \in L({P}_{m})$, the induction hypothesis tells us that $\pr{\pr{{B}_{m, \tau}}{{T}_{m, \tau}}}{\pr{{B}_{m, {\sigma}_{m}}}{{T}_{m, {\sigma}_{m}}}}$ is $\pr{i}{j}$-saturated w.r.t.\@ $c$.
 Since we know that $\pr{{B}_{m+1, \tau}}{{T}_{m+1, \tau}} \leq \pr{{B}_{m, \tau}}{{T}_{m, \tau}}$ and $\left\langle {B}_{m+1, \sigma}, {T}_{m+1, \sigma} \right\rangle \leq \pr{{B}_{m, {\sigma}_{m}}}{{T}_{m, {\sigma}_{m}}}$, we conclude that $\pr{\pr{{B}_{m+1, \tau}}{{T}_{m+1, \tau}}}{\left\langle {B}_{m+1, \sigma}, {T}_{m+1, \sigma} \right\rangle}$ is also $\pr{i}{j}$-saturated w.r.t.\@ $c$.
 This verifies (7).
 
 To verify (6), fix $m' \in \omega$ with $m' < m+1$ and fix $\sigma \in L({P}_{m+1})$.
 Suppose first that $\sigma \in L({P}_{m}) \setminus \{{\sigma}_{m}\}$.
 If $m' < m$, then the induction hypothesis together with the fact that ${T}_{m+1, \sigma} \subseteq {T}_{m, \sigma}$ gives what is needed.
 Now suppose that $m' = m$.
 Then we cannot have ${\sigma}_{m} \subsetneq \sigma$.
 If $\lex{{\sigma}_{m}}{\sigma}$, then $\sigma \in {\GG}_{0}$ and by the definition of ${T}_{m+1, \sigma}$, for each $K \in {T}_{m+1, \sigma}$, ${x}_{m+1} \neq F(K) \ \text{and} \ c({x}_{m+1}, F(K)) = j$.
 Similarly if $\lex{\sigma}{{\sigma}_{m}}$, then $\sigma \in {\GG}_{1}$ and by the definition of ${T}_{m+1, \sigma}$, for each $K \in {T}_{m+1, \sigma}$, ${x}_{m+1} \neq F(K) \ \text{and} \ c({x}_{m+1}, F(K)) = i$.
 This finishes the case when $\sigma \in L({P}_{m}) \setminus \{{\sigma}_{m}\}$.
 Next suppose that $\sigma = {({\sigma}_{m})}^{\frown}{\langle n \rangle}$, for some $n \in \omega$.
 Observe that ${\sigma}_{m'} \in {P}_{m}$ and hence that ${\sigma}_{m'} \neq {({\sigma}_{m})}^{\frown}{\langle k \rangle}$ for any $k \in \omega$.
 Note also that ${\sigma}_{m} \in L({P}_{m})$.
 Furthermore, we know that $\pr{{B}_{m+1, \sigma}}{{T}_{m+1, \sigma}} \leq \pr{{B}_{m, {\sigma}_{m}}}{{T}_{m, {\sigma}_{m}}}$.
 Therefore for any $K \in {T}_{m+1, \sigma}$, $K \cap {B}_{m, {\sigma}_{m}} \in {T}_{m, {\sigma}_{m}}$ and $F(K) = F(K \cap {B}_{m, {\sigma}_{m}})$.
 Now suppose that ${\sigma}_{m'} \subsetneq \sigma$.
 Then ${\sigma}_{m'} \subseteq {\sigma}_{m}$.
 If ${\sigma}_{m'} = {\sigma}_{m}$, then $m = m'$ and by (13) we have that for each $K \in {T}_{m+1, \sigma}$, $F(K) \neq {x}_{m+1}$ and $c({x}_{m+1}, F(K)) = i$, as required.
 So assume that ${\sigma}_{m'} \subsetneq {\sigma}_{m}$.
 Then $m' < m$ and by the induction hypothesis for each $K \in {T}_{m, {\sigma}_{m}}$, $F(K) \neq {x}_{m'+1}$ and $c(F(K), {x}_{m'+1}) = i$.
 Therefore for each $K \in {T}_{m+1, \sigma}$, $F(K) = F(K \cap {B}_{m, {\sigma}_{m}}) \neq {x}_{m'+1}$ and $c(F(K), {x}_{m'+1}) = c(F(K \cap {B}_{m, {\sigma}_{m}}), {x}_{m'+1}) = i$.
 This finishes the case when ${\sigma}_{m'} \subsetneq \sigma$.
 Next assume that $\lex{{\sigma}_{m'}}{\sigma}$.
 Then $\lex{{\sigma}_{m'}}{{\sigma}_{m}}$ and $m' < m$.
 So by the induction hypothesis, for each $K \in {T}_{m, {\sigma}_{m}}$, $F(K) \neq {x}_{m'+1}$ and $c(F(K), {x}_{m'+1}) = j$.
 Therefore for any $K \in {T}_{m+1, \sigma}$, $F(K) = F(K \cap {B}_{m, {\sigma}_{m}}) \neq {x}_{m'+1}$ and $c(F(K), {x}_{m'+1}) = c(F(K \cap {B}_{m, {\sigma}_{m}}), {x}_{m'+1}) = j$.
 Finally assume that $\lex{\sigma}{{\sigma}_{m'}}$.
 Then $\lex{{\sigma}_{m}}{{\sigma}_{m'}}$ and $m' < m$.
 So by the induction hypothesis, for each $K \in {T}_{m, {\sigma}_{m}}$, $F(K) \neq {x}_{m'+1}$ and $c(F(K), {x}_{m'+1}) = i$.
 Therefore for any $K \in {T}_{m+1, \sigma}$, $F(K) = F(K \cap {B}_{m, {\sigma}_{m}}) \neq {x}_{m'+1}$ and $c(F(K), {x}_{m'+1}) = c(F(K \cap {B}_{m, {\sigma}_{m}}), {x}_{m'+1}) = i$.
 This concludes the verification of (6).
 
 Therefore sequences $\langle {x}_{m+1}: m \in \omega \rangle$ and $\seq{F}{m}{\in}{\omega}$ having the required properties can be constructed.
 This finishes the proof of the theorem.
\end{proof}
In the case when $\delta$ is an uncountable strongly compact cardinal, we need a reflection argument telling us that only topological spaces that are members of ${V}_{\delta}$ are relevant.
The argument below is similar to the proof that a stationary set reflects to some ordinal below a strongly compact cardinal. 
\begin{Lemma} \label{lem:reflection}
  Suppose $\delta > \omega$ is a strongly compact cardinal.
  Suppose $X$ is a topological space which is not left-separated and has a point countable base.
  Then there exists a subspace $Y \subseteq X$ with $\lc Y \rc < \delta$ which is not left-separated and has a point-countable base.
\end{Lemma}
\begin{proof}
  Let $\TT$ be the topology on $X$ with a point-countable base $\BB \subseteq \TT$.
  Suppose that every subspace of $X$ with size less than $\delta$ is left-separated.
  By the fact that $\delta$ is strongly compact, we can find an elementary embedding $j: V \rightarrow M$ and a set $Y \in M$ such that the critical point of $j$ is $\delta$, $j''X \subseteq Y \subseteq j(X)$, and ${\lc Y \rc}^{M} < j(\delta)$.
  By our hypothesis, working in $M$, we find that $Y$ can be left-separated.
  So the following statement holds in $M$: there is an ordinal $\alpha$ and a bijection $f: \alpha \rightarrow Y$ such that for each $\xi < \alpha$, $\{f(\zeta): \zeta < \xi\}$ is closed relative to $Y$.
  Now in $V$ define a well-ordering of $X$ as follows.
  For any $x, x' \in X$, $x' \prec x$ if and only if ${f}^{-1}(j(x')) < {f}^{-1}(j(x))$.
  The order $\prec$ is clearly a well-ordering of $X$.
  Now fix $x \in X$.
  We must check that $I = \{x' \in X: x' \prec x\}$ is closed in $X$.
  Suppose $x'' \in X$ belongs to the closure of $I$.
  Since $X$ has a point-countable base, we can find a sequence $\langle {x}_{n}': n \in \omega \rangle$ converging to $x''$ such that $\forall n \in \omega\[{x}_{n}' \in I\]$.
  By elementarity, working in $M$, we have that $j(\langle {x}_{n}': n \in \omega \rangle)$ converges to $j(x'')$ in $\pr{j(X)}{j(\TT)}$, which is a topological space according to $M$.
  Note that $j(\langle {x}_{n}': n \in \omega \rangle) = \langle j({x}_{n}'): n \in \omega \rangle$.
  Put $\xi = {f}^{-1}(j(x))$.
  Then $\{j({x}_{n}'): n \in \omega\} \subseteq \{f(\zeta): \zeta < \xi\}$.
  We know that $\{f(\zeta): \zeta < \xi\}$ is a closed subset of $Y$ according to $M$.
  Therefore $j(x'') \in \{f(\zeta): \zeta < \xi\}$, whence $x'' \prec x$, and so $x'' \in I$.
  This shows that $I$ is closed in $X$.
  Thus $\prec$ witnesses that $X$ can be left-separated, contradicting our hypothesis on $X$.
  Therefore there must exist a subspace $Y \subseteq X$ with $\lc Y \rc < \delta$ which cannot be left-separated.
  It is easy to see that $\BB' = \{U \cap Y: U \in \BB\}$ is a point-countable base for the subspace topology on $Y$.
  Hence $Y$ is as required.
\end{proof}
Now we can state the following corollaries to Theorem \ref{thm:main} and Lemma \ref{lem:reflection}.
Corollary \ref{cor:reals} establishes Theorem \ref{thm:metricspaces} for all uncountable sets of reals.
Note that a single Woodin cardinal suffices for this special case of Theorem \ref{thm:metricspaces} as every set of reals is a member of ${V}_{\delta}$ when $\delta$ is the least Woodin cardinal.
Corollary \ref{cor:woodin} establishes Theorem \ref{thm:general} which, as we pointed out in Section \ref{intro2}, implies Theorem \ref{thm:metricspaces}.
\begin{Cor} \label{cor:reals}
  Suppose there exists at least one Woodin cardinal or one uncountable strongly compact cardinal.
  Then for every uncountable set of reals $X$, every $l \in \omega$, and every coloring $c: {\[X\]}^{2} \rightarrow l$, there exists $Y \subseteq X$ such that $Y$ is homeomorphic to the rationals and $c$ realizes at most two colors on $Y$. 
\end{Cor}
We would like to note that it is easy to modify the proof of Theorem \ref{thm:main} to show that the conclusion of Corollary \ref{cor:reals} also holds if there is a precipitous ideal on ${\omega}_{1}$.
It is not known at present whether any large cardinal hypothesis proves the existence of a precipitous ideal on ${\omega}_{1}$.
However the existence of a precipitous ideal on ${\omega}_{1}$ is equal in consistency strength to the existence of one measurable cardinal (see \cite{Je}), which is considerably lower in consistency strength than the existence of one Woodin cardinal.
Hence a measurable cardinal puts an upper bound on the consistency strength of the statement that the $2$-dimensional Ramsey degree of $\Q$ within the class of all uncountable sets of real numbers is $2$.
\begin{Cor} \label{cor:woodin}
  Suppose there exists a proper class of Woodin cardinals or one uncountable strongly compact cardinal.
  Then for every regular topological space $\pr{X}{\TT}$ which is not left-separated and has a point-countable base, every $l \in \omega$, and every coloring $c: {\[X\]}^{2} \rightarrow l$, there exists $Y \subseteq X$ such that $Y$ is homeomorphic to the rationals and $c$ realizes at most two colors on $Y$. 
\end{Cor}
\begin{proof}
  If there is a proper class of Woodin cardinals, then for every topological space $\pr{X}{\TT}$, there is a Woodin cardinal $\delta$ so that $\pr{X}{\TT} \in {V}_{\delta}$.
  Hence the conclusion of the corollary immediately follows from Theorem \ref{thm:main}.
  
  Next, suppose that $\delta > \omega$ is a strongly compact cardinal. 
  Suppose for a contradiction that there is a counterexample to the conclusion of the corollary.
  Then by Lemma \ref{lem:reflection}, and by the fact that a subspace of a regular space is regular, we can find a counterexample $\pr{X}{\TT} \in {V}_{\delta}$, together with $l \in \omega$ and a coloring $c: {\[X\]}^{2} \rightarrow l$.
  However this contradicts Theorem \ref{thm:main}.  
\end{proof}
\def\polhk#1{\setbox0=\hbox{#1}{\ooalign{\hidewidth
  \lower1.5ex\hbox{`}\hidewidth\crcr\unhbox0}}}
\providecommand{\bysame}{\leavevmode\hbox to3em{\hrulefill}\thinspace}
\providecommand{\MR}{\relax\ifhmode\unskip\space\fi MR }
\providecommand{\MRhref}[2]{%
  \href{http://www.ams.org/mathscinet-getitem?mr=#1}{#2}
}
\providecommand{\href}[2]{#2}

\end{document}